\documentclass[11pt]{amsart}
\usepackage{amsfonts,amsthm,amsmath}
\usepackage{natbib}
\usepackage{hyperref}

\newtheorem{theorem}{Theorem}[section]
\newtheorem{lemma}{Lemma}[section]

\newcommand{\vague}{\stackrel{v}{\longrightarrow}}
\newcommand{\weak}{\stackrel{w}{\longrightarrow}}
\newcommand{\prob}{\stackrel{P}{\longrightarrow}}

\newcommand{\one}{{\bf 1}}
\newcommand{\reals}{{\mathbb R}}
\newcommand{\bbr}{\reals}
\newcommand{\vep}{\varepsilon}

\newcommand{\bbrdcomp}{\overline{\bbr^d}}
\newcommand{\nchoosek}{\left(\begin{array}cn\\k\end{array}\right)}

\def\Var{{\rm Var}}

\numberwithin{equation}{section}

\begin{document}

\title[Truncated heavy tails]{Effect of truncation on large deviations for heavy-tailed random vectors}
\author[A. Chakrabarty]{Arijit Chakrabarty}
\address{Statistics and Mathematics Unit,
Indian Statistical Institute,
7 S.J.S. Sansanwal Marg,
New Delhi 110016, India}
\email{arijit@isid.ac.in}

\begin{abstract}
This paper studies the effect of truncation on the large deviations behavior of the partial sum of a triangular array coming from a truncated power law model. Each row of the triangular array consists of i.i.d. random vectors, whose distribution matches a power law on a ball of radius going to infinity, and outside that it has a light-tailed modification. The random vectors are assumed to be
$\bbr^d$-valued. It turns out that there are two regimes depending
on the growth rate of the truncating threshold, so that in one
regime, much of the heavy tailedness is retained, while in the
other regime, the same is lost.
\end{abstract}
 \keywords{ heavy
tails, truncation, regular variation,  large deviation\vspace{.5ex}}
\thanks{Research partly supported by the NSF grant ``Graduate and Postdoctoral Training in Probability and its Applications'' at
Cornell University.}
\maketitle

\section{Introduction} \label{sec:intro}

This paper answers the question of the extent to which truncated
heavy-tailed random vectors behave like heavy-tailed random
vectors that are not truncated, from the point of view of large
deviations behavior. There are lot of situations where a power law
is a good fit, and at the same time the quantity of interest is
physically bounded above. As a natural model for such phenomena, we consider a
truncated heavy-tailed distribution - a distribution that matches
a power law on a ball with ``large'' radius, centered at the
origin, and outside that the tail decays significantly faster or
simply vanishes. It is obvious that if the truncating threshold is
fixed, then as the sample size goes to infinity, any effect of the
heavy-tailed distribution that we started with will eventually
wash out. Thus, any interesting analysis of such a system should
necessarily let the truncating threshold go to infinity along with
the sample size. Answering the question posed above demands a
systematic study of the relation between the growth rate of the
truncating threshold and the asymptotic properties of the
truncated heavy-tailed model which we now proceed to define
formally. This question has previously been addressed in the
literature from a different angle, that of the central limit
theorem; see \cite{chakrabarty:samorodnitsky:2009} and
\cite{chakrabarty:2010}.

A random variable $H$ that takes values in $\bbr^d$ is  heavy-tailed or has a power law, if there is a non-null Radon measure $\mu$ on $\bbr^d\setminus\{0\}$ so that there is a sequence $a_n$ going to infinity satisfying
\begin{equation}\label{ld.defn}
nP(a_n^{-1}H\in\cdot)\vague\mu(\cdot)
\end{equation}
on $\bbrdcomp\setminus\{0\}$. Here $\bbrdcomp$ is a compact set obtained by adding to $\bbr^d$ a ball of infinite radius centered at origin and the measure $\mu$ is extended to the former by $\mu(\bbrdcomp\setminus\bbr^d)=0$. It can be shown that \eqref{ld.defn} implies that there exists $\alpha>0$ such that for any Borel set $A\subset B$ and $c>0$,
$\mu(cA)=c^{-\alpha}\mu(A)\,.$
This is the definition of regularly varying tail with index $\alpha$ used by \cite{resnick:1987} and \cite{hult:lindskog:mikosch:samorodnitsky:2005}. Since the truncating threshold changes with the sample size, we have a triangular array. The $n$-th row of the array, comprises $n$ i.i.d. random vectors denoted by $X_{n1},\ldots,X_{nn}$. For $1\le j\le n$, the observation $X_{nj}$, whose distribution should be thought of as the truncation of a power tail, is defined by
\begin{equation}\label{e:the.model}
X_{nj}:=H_j\one\left(\|H_j\|\le  M_n\right)
+\frac{H_j}{\|H_j\|}(M_n+L_j)\one\left(\|H_j\|>M_n\right)\,.
\end{equation}
Here $(M_n)$ is a sequence of numbers going to infinity, $H_1,H_2,\ldots$ are i.i.d.
copies of $H$ that satisfies \eqref{ld.defn}, and $(L,L_1,L_2,\ldots)$ is a sequence of i.i.d. nonnegative random variables. We assume that the families $(H,H_1,H_2,\ldots)$ and $(L,L_1,L_2,\ldots)$ are  independent. In \eqref{e:the.model}, $M_n$ denotes the level of truncation. The distribution of the random variable $L$ represents the modification of the model \eqref{e:the.model} outside the ball of radius $M_n$. We chose to formulate the results in such a way that all of them will be true in the case when $L$ is identically zero. However, almost all the results are true under milder hypothesis like existence of some exponential moment. The assumption on $L$ will vary from result to result and will be stated as we go along.
We would like to mention at this point that the model \eqref{e:the.model} makes the modification outside the ball of radius $M_n$ radially identical, an assumption made for the sake of simplicity. An interesting extension, which we leave aside for future investigation, would be to multiply $L_j$ by a function of $H_j/\|H_j\|$.

The motivation of this paper is based on the fact that the notion of heavy-tail as defined in \eqref{ld.defn} is closely related to large deviation results for random walks with heavy-tailed step size. Such studies in one dimension date back to \cite{heyde:1968}, \cite{nagaev:1969b}, \cite{nagaev:1969a}, \cite{nagaev:1979} and \cite{cline:hsing:1991}, among others; a survey on this topic can be found in Section 8.6 in \cite{embrechts:kluppelberg:mikosch:1997} and \cite{mikosch:nagaev:1998}. More recently, the functional version of large deviation principles for heavy-tailed $\bbr^d$ valued random variables has been taken up by \cite{hult:lindskog:mikosch:samorodnitsky:2005}. There, it is shown among other things, that if $H_1,H_2,\ldots$ are i.i.d. copies of $H$ that satisfies \eqref{ld.defn}, then
\begin{equation}\label{intro.ld.ht}
\frac{P\left(\lambda_n^{-1}\sum_{j=1}^nH_j\in\cdot\right)}{nP(\|H\|>\lambda_n)}\vague\frac{\mu(\cdot)}{\mu(B_1^c)}\,,
\end{equation}
where $\lambda_n$ is a sequence satisfying $\lambda_n^{-1}\sum_{j=1}^nH_j\prob0$ and in addition
\begin{eqnarray*}
\lambda_n&\gg&\sqrt{n^{1+\gamma}}\mbox{ for some }\gamma>0,\mbox{ if }\alpha=2\\
\lambda_n&\gg&\sqrt{n\log n},\mbox{ if }\alpha>2\,,
\end{eqnarray*}
and for $r\ge0$, $B_r:=\{x\in\bbr^d:\|x\|\le r\}$ denotes the closed ball of radius $r$ centered at the origin. (In the above equation, ``$v_n\gg u_n$'' means that
$$
\lim_{n\to\infty}\frac{u_n}{v_n}=0\,.
$$
Throughout the paper, ``$\gg$'' will be used as a shorthand for the above, and ``$\ll$'' for the obvious opposite.)
Motivated by this, we ask the question ``When does the model \eqref{e:the.model} retain the heavy-tailedness so that the behavior is similar to that in \eqref{intro.ld.ht}?''
The conclusion of \cite{chakrabarty:samorodnitsky:2009} was that the central limit behavior was completely determined by the truncation regime defined as follows: the tails in
the model \eqref{e:the.model} are called
\begin{equation} \label{e:regimes}
\begin{array}{ll}
\text{truncated softly} & \text{if} \ \lim_{n\to\infty}nP\left( \|
H\|>M_n\right) = 0\,,\\
\text{truncated hard} & \text{if} \ \lim_{n\to\infty}nP\left( \|
H\|>M_n\right) = \infty\,.
\end{array}
\end{equation}
Our approach to answering the above mentioned question lies in studying the large deviation behavior of the partial sum in both regimes - soft and hard  truncation, as defined in \eqref{e:regimes}. Of course, there is an intermediate regime where the limit exists, and is finite and positive. Unfortunately, the author has not been able to solve the large deviations for that regime. The above mentioned reference studies the central limit behavior for that regime.

The paper is organized as follows.  The large deviation principles for
the truncated heavy-tailed random variables is studied in the soft
truncation and hard truncation regimes, as defined in
\eqref{e:regimes}, in Sections \ref{sec:ld.st} and \ref{sec:ld.ht}
respectively. The conclusions of the paper are summarized in
Section \ref{sec:con}.

\section{Large deviations: the soft truncation regime}\label{sec:ld.st} In this section, we study the behavior of the large deviation probabilities for sums of truncated heavy-tailed random variables, when the truncation is soft. Let $H$ be a $\bbr^d$ valued random variable satisfying \eqref{ld.defn} for some sequence $a_n$ going to infinity and a non-null Radon measure $\mu$ on $\bbrdcomp$ with $\mu(\bbrdcomp\setminus\bbr^d)=0$. 
It is well known that for such a $H$, $P(\|H\|>\cdot)$ is regularly varying with index $-\alpha$ for some $\alpha>0$.
We further assume that if $\alpha=1$ then $H$ has a symmetric distribution and if $\alpha>1$ then $E(H)=0$.
The triangular array $\{X_{nj}:1\le j\le n\}$ is as defined in \eqref{e:the.model}, where $H_1,H_2,\ldots$ are i.i.d. copies of $H$, $M_n$ is a sequence of positive numbers going to $\infty$, $L,L_1,L_2,\ldots$ are i.i.d. $[0,\infty)$ valued random variables independent of $H,H_1,H_2,\ldots$ and $\|\cdot\|$ denotes the $L^2$ norm on $\bbr^d$, {\it i.e.},  for $x=(x_1,\ldots,x_d)\in\bbr^d$,
\begin{equation}\label{ld.norm}
\|x\|:=\left(\sum_{j=1}^dx_j^2\right)^{1/2}\,.
\end{equation}
We shall study large deviations for the row sum $S_n$, defined by
$$
S_n:=\sum_{j=1}^nX_{nj}\,.
$$

For this section, we assume that  $M_n$ goes to $\infty$ fast enough so that
$$\lim_{n\to\infty}nP(\|H\|>M_n)=0\,,$$
which is clearly equivalent to
$$
M_n\gg a_n\,,
$$
where $a_n$ is that satisfying \eqref{ld.defn}.
We assume in addition that
\begin{equation}\label{alpha2}
\lim_{n\to\infty}M_n/\sqrt{n^{1+\gamma}}=\infty\mbox{ for some }\gamma>0,\mbox{ if }\alpha=2\,,
\end{equation}
and
\begin{equation}\label{alphag2}
\lim_{n\to\infty}M_n/\sqrt{n\log M_n}=\infty,\mbox{ if }\alpha>2\,.
\end{equation}
Define
\begin{equation}\label{ld.eq5}
b_n:=\left\{
\begin{array}{ll}
\inf\{x:P(\|H\|>x)\le n^{-1}\},&\alpha<2\\
\sqrt{n^{1+\gamma}},&\alpha=2\\
\sqrt{n\log n},&\alpha>2\,,
\end{array}\right.
\end{equation}
where $\gamma$ is same as that in \eqref{alpha2}. Clearly, $1\ll b_n\ll M_n$ and ${\mathcal L}(b_n^{-1}S_n)$ is a tight sequence. The following result, which is an easy consequence of Lemma 2.1 in \cite{hult:lindskog:mikosch:samorodnitsky:2005}, describes the large deviation behavior of $\lambda_n^{-1}S_n$ where $b_n\ll\lambda_n\ll M_n$.

\begin{theorem}\label{ld.t1} In the soft truncation regime, if $\lambda_n$ is any sequence of positive numbers satisfying $b_n\ll\lambda_n\ll M_n$,
then, as $n\longrightarrow\infty$,
$$\frac{P(\lambda_n^{-1}S_n\in\cdot)}{nP(\|H\|>\lambda_n)}\vague\frac{\mu(\cdot)}{\mu(B_1^c)}$$
on $\bbrdcomp\setminus\{0\}$. Recall that for all $r\ge0$, $B_r$ denotes the closed ball of radius $r$, centered at the origin.
\end{theorem}

\begin{proof}
Fix a sequence $\lambda_n$ satisfying the hypotheses. The assumption
that $\lambda_n\gg b_n$ implies that $\lambda_n^{-1}S_n\prob0\,.$ By
Lemma 2.1 in \cite{hult:lindskog:mikosch:samorodnitsky:2005}, it
follows that
$$
\frac{P\left(\lambda_n^{-1}\sum_{j=1}^n H_j\in\cdot\right)}{nP(\|H\|>\lambda_n)}\vague\frac{\mu(\cdot)}{\mu(B_1^c)}
$$
on $\bbrdcomp\setminus\{0\}$. Note that
\begin{eqnarray*}
&&\sup_{A\subset\bbr^d}\left|P(\lambda_n^{-1}S_n\in A)-P\left(\lambda_n^{-1}\sum_{j=1}^nH_j\in A\right)\right|\\
&\le&nP(\|H\|>M_n)\\
&=&o(nP(\|H\|>\lambda_n))\,,
\end{eqnarray*}
the last equality following from the assumption that $\lambda_n\ll M_n$. This completes the proof.
\end{proof}

Before stating the next result we need some preliminaries. Define
\begin{equation}\label{defS}
{\mathcal S}:=\{x\in\bbr^d:\|x\|=1\}\,,
\end{equation}
and a probability measure $\sigma$ on $\mathcal S$ by
\begin{equation}\label{defsigma}
\sigma(A):=\frac1{\mu(B_1^c)}\mu\left(\left\{x\in \bbr^d:\|x\|\ge1,\frac x{\|x\|}\in A\right\}\right)\,.
\end{equation}
Notice that $\sigma$ is the measure satisfying
$$
\frac{\mu((r,\infty)\times A)}{\mu((1,\infty)\times {\mathcal S})}=r^{-\alpha}\sigma(A),\,r>0,A\subset{\mathcal S}\,,
$$
which is a consequence of the scaling property satisfied by $\mu$, mentioned below \eqref{ld.defn}.
It is easy to see that \eqref{ld.defn} implies
\begin{equation}\label{ld.spectral}
{P\left(\frac H{\|H\|}\in\cdot\bigg|\|H\|>t\right)}\weak\sigma(\cdot)
\end{equation}
as $t\longrightarrow\infty$, weakly on $\mathcal S$.

For $k\ge1$, we define a measure $\nu^{(k)}$ on $\bbr^d\setminus B_{k-1}$ by
$$\nu^{(k)}(A):=\int\cdots\int1\left(\sum_{j=1}^kx_j\in A\right)\nu(dx_1)\ldots\nu(dx_k)\,,$$
where
\begin{equation}\label{ld.t2.eq1}
\nu(A):=\frac{\mu(A\cap B_1)}{\mu(B_1^c)}+\sigma(A\cap{\mathcal S})\,.
\end{equation}
Extend $\nu^{(k)}$ to $\bbrdcomp\setminus B_{k-1}$ by putting $\nu^{(k)}(\bbrdcomp\setminus\bbr^d)=0$. Let us record some properties of this measure. First, notice that $\nu^{(k)}$ is a Radon measure, that is, $\nu^{(k)}(B_r^c)<\infty$ for all $r>k-1$, which follows from the fact that $\nu$ puts finite measure on the set $B_{r-k+1}^c$, and the observation that
\begin{eqnarray*}
&&\nu^{(k)}(B_r^c)\\
&=&\int_{\{1\ge\|x_1\|>r-k+1\}}\ldots\int_{\{1\ge\|x_k\|>r-k+1\}}\one\left(\sum_{j=1}^kx_j\in
B_r^c\right)\nu(dx_1)\ldots\\
&&\,\,\,\,\ldots\nu(dx_k)\,,
\end{eqnarray*}
the equality following because $\nu(B_1^c)=0$. The next observation is that
$$
\nu^{(k)}(B_k^c)=0\,,
$$
which follows trivially from the definition. Finally, observe that
$$
\nu^{(1)}=\nu\,.
$$

The next result, Theorem \ref{ld.t2}, describes the large deviation behavior of $M_n^{-1}S_n$. The reason we call this a large deviation result is the following. This result, for example, shows that for all $r\in(k-1,k)$ such that $\nu^{(k)}(\{x\in\bbr^d:\|x\|=r\})=0$ (which is in fact true for all but countably many $r$'s in $(k-1,k)$), there is some $C_r\in(0,\infty)$ so that, as $n\longrightarrow\infty$,
$$P(\|S_n\|>rM_n)\sim C_r\{nP(\|H\|>M_n)\}^k\,.$$

\begin{theorem}\label{ld.t2} Suppose $k\ge1$ and that
\begin{equation}\label{ld.t2.eq2}
P(L>x)=o(P(\|H\|>x)^{k-1})
\end{equation}
as $x\longrightarrow\infty$. Then, in the soft truncation regime, as $n\longrightarrow\infty$,
 \begin{eqnarray*}
\frac{P(M_n^{-1}S_n\in\cdot)}{\{nP(\|H\|>M_n)\}^k}
 &\stackrel v\longrightarrow&\frac1{k!}\nu^{(k)}(\cdot)
 \end{eqnarray*}
 on $\bbrdcomp\setminus B_{k-1}$.
\end{theorem}

Before going to the proof, let us closely inspect the statement of the above result. Fix $k\ge1$. Since $\nu^{(k)}$ does not charge anything outside $B_k$ and the vague convergence happens on $\bbrdcomp\setminus B_{k-1}$, assume that
\begin{equation}\label{eq.new2}
A\subset B_k\setminus B_{k-1+\vep}\,,
\end{equation}
for some $\vep>0$. All that Theorem \ref{ld.t2} says is
$$
\frac1{k!}\nu^{(k)}(int(A))\le\liminf_{n\to\infty}\frac{P(M_n^{-1}S_n\in A)}{\{nP(\|H\|>M_n)\}^k}\le
$$
$$
\limsup_{n\to\infty}\frac{P(M_n^{-1}S_n\in A)}{\{nP(\|H\|>M_n)\}^k}\le\frac1{k!}\nu^{(k)}(cl(A))\,,
$$
where $int(\cdot)$ and $cl(\cdot)$ denote the interior and the closure of a set respectively.

The proof of Theorem \ref{ld.t2} is based on the idea that for $M_n^{-1}S_n$ to belong to a set $A$ satisfying \eqref{eq.new2}, it is ``necessary and sufficient'' that $M_n^{-1}\sum_{u=1}^kX_{nj_u}$ belongs to $A$ for at least one tuple $1\le j_1<\ldots<j_k\le n$,
where $X_{nj}$'s are as defined in \eqref{e:the.model}. This idea is similar to the idea in the proof of Lemma 2.1 in \cite{hult:lindskog:mikosch:samorodnitsky:2005}, that $S_n$ is large ``if and only if'' exactly one of the summands is large. The above heuristic statement is equivalent to
\begin{eqnarray}
&&P(M_n^{-1}S_n\in A)\nonumber\\
&\sim&P\left(\bigcup_{1\le j_1<\ldots<j_k\le n}\left\{M_n^{-1}\sum_{u=1}^kX_{nj_u}\in A\right\}\right)\nonumber\\
&\sim&\nchoosek P\left(M_n^{-1}\sum_{j=1}^kX_{nj}\in A\right)\nonumber\\
&=&\nchoosek \int\ldots\int\one\left(\sum_{j=1}^kx_j\in A\right)P(M_n^{-1}X_{n1}\in dx_1)\label{eq.new}\\
&&\,\,\,\,\,\,\,\ldots P(M_n^{-1}X_{nk}\in dx_k)\,.\nonumber
\end{eqnarray}
Again heuristically,
$$
P(M_n^{-1}X_{n1}\in dx)\sim nP(\|H\|>M_n)\nu(dx)\,,
$$
a formal statement of which is precisely the content of Lemma \ref{ld.l11} below. Using this, it can be argued that
\begin{eqnarray*}
&&\int\ldots\int\one\left(\sum_{j=1}^kx_j\in A\right)P(M_n^{-1}X_{n1}\in dx_1)\ldots P(M_n^{-1}X_{nk}\in dx_k)\\
&\sim&\int\ldots\int\one\left(\sum_{j=1}^kx_j\in A\right)\nu(dx_1)\ldots\nu(dx_k)\\
&=&\nu^{(k)}(A)\,.
\end{eqnarray*}
The above, in view of \eqref{eq.new}, shows the statement of Theorem \ref{ld.t2}. These ideas, in fact, constitute the crux of the rigorous proof. For the latter, we shall need the following
lemmas.

 \begin{lemma}\label{ld.l11}As $t\longrightarrow\infty$,
 $$\frac{P(X^t/t\in\cdot)}{P(\|H\|>t)}\stackrel v\longrightarrow\nu(\cdot)$$
 on $\bbrdcomp\setminus\{0\}$, where, for $t>0$,
 $$X^t:=H\one\left({\|H\|\le t}\right)+(t+L)\frac H{\|H\|}\one\left({\|H\|>t}\right)\,.$$
 \end{lemma}

 \begin{proof} Since for all $\epsilon>0$, $\nu$ restricted to $B_\epsilon^c$ is a finite measure, it suffices to show that for $\epsilon\in(0,1)$,
\begin{equation}\label{ld.l11.eq11}
\lim_{t\to\infty}\frac{P(X^t/t\in
B_\epsilon^c)}{P(\|H\|>t)}=\nu(B_\epsilon^c)\,,
\end{equation}
and that for $A\subset{\mathbb R}^d$ which is closed and bounded away from zero,
 \begin{equation}\label{ld.l11.eq1}
\limsup_{t\rightarrow\infty}\frac{P(X^t/t\in
A)}{P(\|H\|>t)}\le\nu(A)\,.
\end{equation}
For \eqref{ld.l11.eq11}, note that
\begin{eqnarray*}
\lim_{t\to\infty}\frac{P(X^t/t\in
B_\epsilon^c)}{P(\|H\|>t)}&=&\lim_{t\to\infty}\frac{P(H/t\in
B_\epsilon^c)}{P(\|H\|>t)}\\
&=&\nu(B_\epsilon^c)\,,
\end{eqnarray*}
where the second equality follows from the fact that
\begin{equation}\label{new.eq1}
\frac{P(H/t\in\cdot)}{P(\|H\|>t)}\vague\frac{\mu(\cdot)}{\mu(B_1^c)}
\end{equation}
in $\bbrdcomp\setminus\{0\}$, which is a consequence of \eqref{ld.defn}, and that $B_\epsilon^c$ is a $\mu$-continuous set. For \eqref{ld.l11.eq1}, fix an $A\subset\bbr^d$ which is closed and bounded away from zero. Define a function
$T$ from ${\mathbb R}^d\setminus\{0\}$ to $\mathcal S$ by
$T(x)=\frac x{\|x\|}$.
Since $A$ is closed,
 $$\bigcap_{\epsilon>0}T(A\cap(B_{1+\epsilon}\setminus int(B_{1-\epsilon})))=A\cap\mathcal S\,.$$
 Thus, for fixed $\delta>0$ there is $\epsilon>0$ so that
 $$\sigma\left(T(A\cap(B_{1+\epsilon}\setminus int(B_{1-\epsilon})))\right)\le\sigma(A\cap{\mathcal S})+\delta\,.$$
 Define
 $$\tilde A:=T(A\cap(B_{1+\epsilon}\setminus int(B_{1-\epsilon})))\,.$$
Since $A\cap(B_{1+\epsilon}\setminus int(B_{1-\epsilon}))$ is
compact and $T$ is continuous, $\tilde A$ is compact and hence
closed. Note that
$$P(X^t/t\in A)\le$$
$$P(X^t/t\in A\cap B_{1-\epsilon})+P(X^t/t\in A\cap(B_{1+\epsilon}\setminus int(B_{1-\epsilon}))+P(\|X^t\|\ge(1+\epsilon)t)\,.$$
Clearly
$$P(X^t/t\in A\cap B_{1-\epsilon})=P(H/t\in A\cap B_{1-\epsilon})$$
and hence by \eqref{new.eq1}, it follows that
$$\limsup_{t\rightarrow\infty}\frac{P(X^t/t\in A\cap B_{1-\epsilon})}{P(\|H\|>t)}\le\frac{\mu(A\cap B_1)}{\mu(B_1^c)}\,.$$
It is also clear that, as $t\longrightarrow\infty$,
$$P(\|X^t\|\ge(1+\epsilon)t)=o\left(P(\|H\|>t)\right)\,.$$
Note that
\begin{eqnarray*}
P\left(X^t/t\in A\cap(B_{1+\epsilon}\setminus
int(B_{1-\epsilon}))\right)&\le&P(H/\|H\|\in\tilde A,\|H\|\ge(1-\epsilon)t)\,.
\end{eqnarray*}
Since $\tilde A$ is closed, by \eqref{ld.spectral} and the fact that $P(\|H\|>\cdot)$ is regularly varying with index $-\alpha$, it follows that
\begin{eqnarray*}
\limsup_{t\rightarrow\infty}\frac{P(H/\|H\|\in\tilde A,\|H\|\ge(1-\epsilon)t)}{P(\|H\|>t)}&\le&(1-\epsilon)^{-\alpha}\sigma(\tilde A)\\
&\le&(1-\epsilon)^{-\alpha}(\sigma(A\cap\mathcal S)+\delta)\,.
\end{eqnarray*}
Since $\epsilon$ and $\delta$ can be chosen to be arbitrarily
small, this shows (\ref{ld.l11.eq1}) and thus completes the
proof.
 \end{proof}

The next lemma studies the asymptotics of the sum of a fixed number ($k$) of random variables in the triangular array $\{X_{nj}:1\le j\le n\}$, as the row index ($n$) goes to infinity.

\begin{lemma}\label{ld.l12}Suppose that \eqref{ld.t2.eq2} holds. Then,
 \begin{eqnarray*}
\frac{P\left(M_n^{-1}\sum_{j=1}^kX_{nj}\in\cdot\right)}{P(\|H\|>M_n)^k}
 &\stackrel v\longrightarrow&\nu^{(k)}(\cdot)\,,
 \end{eqnarray*}
 on $\bbrdcomp\setminus B_{k-1}$.
\end{lemma}

\begin{proof}
Fix a $\nu^{(k)}$ continuity set $A\subset B_\delta^c$ for some
$k-1<\delta<k$. Fix $\epsilon>0$ so that $(k-1)(1+\epsilon)<\delta$. Clearly,
 \begin{eqnarray*}
 && P\left(M_n^{-1}\sum_{j=1}^kX_{nj}\in A,\|X_{nj}\|\le (1+\epsilon)M_n,1\le j\le k\right)\\
&\le&P\left(M_n^{-1}\sum_{j=1}^kX_{nj}\in A\right)\\
&\le& P\left(M_n^{-1}\sum_{j=1}^kX_{nj}\in A,\|X_{nj}\|\le (1+\epsilon)M_n,1\le j\le k\right)\\
&&+kP(L>\epsilon M_n)P(\|H\|>M_n)\,.
\end{eqnarray*}
By the assumption on $L$, it follows that
\begin{eqnarray*}
P(L>\epsilon M_n)&=&o(P(\|H\|>\epsilon M_n)^{k-1})\\
&=&o(P(\|H\|>M_n)^{k-1})\,.
\end{eqnarray*}
Since $A\subset B_\delta^c$ where $\delta>(k-1)(1+\epsilon)$,
\begin{eqnarray}
&& P\left(M_n^{-1}\sum_{j=1}^kX_{nj}\in A,\|X_{nj}\|\le (1+\epsilon)M_n,1\le j\le k\right)\nonumber\\
&&=\int_{\{\eta<\|x_1\|\le1+\epsilon\}}\cdots\int_{\{\eta<\|x_k\|\le1+\epsilon\}}\one\left(\sum_{j=1}^kx_j\in A\right)\nonumber\\
&&\,\,\,\,\,P(M_n^{-1}X_{n1}\in dx_1)\ldots P(M_n^{-1}X_{nk}\in dx_k)\nonumber\\
&&=\int_{\{\|x_1\|\le1+\epsilon\}}\cdots\int_{\{\|x_k\|\le1+\epsilon\}}\one\left(\sum_{j=1}^kx_j\in A\right)P_n(dx_1)\ldots P_n(dx_k)\,,\label{ld.l12.eq1}
\end{eqnarray}
where $\eta:=\delta-(k-1)(1+\epsilon)>0$ and $P_n(\cdot)$ denotes the restriction of $P(M_n^{-1}X_{n1}\in\cdot)$ to $\bbr^d\setminus B_\eta$. Let $\tilde\nu$ denote the restriction of $\nu$ to $\bbr^d\setminus B_\eta$. Then, by Lemma \ref{ld.l11}, as $n\longrightarrow\infty$,
$$\frac{P_n(\cdot)}{P(\|H\|>M_n)}\weak\tilde\nu(\cdot)$$
on $\bbr^d\setminus B_\eta$. Thus,
$$\frac{P_n(dx_1)\ldots P_n(dx_k)}{P(\|H\|>M_n)^k}\weak\tilde\nu(dx_1)\ldots\tilde\nu(dx_k)$$
on $(\bbr^d\setminus B_\eta)^k$, as $n\longrightarrow\infty$.
Consider the function  $f:\bbr^{d\times k}\longrightarrow\bbr$ defined by
$$f(x_1,\ldots,x_k)=\one(\|x_1\|\le1+\epsilon)\ldots\one(\|x_k\|\le1+\epsilon)\one\left(\sum_{j=1}^kx_j\in A\right)\,.$$
The set of discontinuities of $f$ is contained in
$$\bigcup_{j=1}^k\{(x_1,\ldots,x_k):\|x_j\|=1+\epsilon\}\cup\left\{(x_1,\ldots,x_k):\sum_{j=1}^kx_j\in\partial A\right\}\,.$$
The product measure ${\tilde\nu}^k$ gives zero measure to this set because $\nu$ (and hence $\tilde\nu$) does not charge anything outside $B_1$ and the set $A$ has been chosen to satisfy
$$\int\ldots\int\one\left(\sum_{j=1}^kx_j\in\partial A\right)\nu(dx_1)\ldots\nu(dx_k)=0\,.$$
Thus, as $n\longrightarrow\infty$, the right hand side of \eqref{ld.l12.eq1} is asymptotically equivalent to
$$P(\|H\|>M_n)^k\int_{\{\|x_1\|\le1+\epsilon\}}\cdots\int_{\{\|x_k\|\le1+\epsilon\}}\one\left(\sum_{j=1}^kx_j\in A\right)\tilde\nu(dx_1)\ldots\tilde\nu(dx_k)\,,$$
which is same as $P(\|H\|>M_n)^k\nu^{(k)}(A)$.
This completes the proof.
\end{proof}

We shall also need the following result, which has been proved in\\
\cite{prokhorov:1959}.

\begin{lemma}\label{ld.l1} If $X_1,\ldots,X_N$ are i.i.d. $\bbr$-valued independent random variables with $|X_i|\le C$ a.s. where $0<C<\infty$, then, for $\lambda>0$,
$$P(S_N-ES_N>\lambda)\le\exp\left\{-\frac\lambda{2C}\sinh^{-1}\frac{C\lambda}{2\Var(S_N)}\right\}\,,$$
where
$$S_N:=\sum_{i=1}^NX_i\,.$$
\end{lemma}

\begin{proof}[Proof of Theorem \ref{ld.t2}] We shall show that for every $\nu^{(k)}$-continuous set $A\subset\bbr^d\setminus B_\delta$ for some $\delta>k-1$,
\begin{equation}\label{ld.t1.eq3}
\lim_{n\to\infty}\frac{P(M_n^{-1}S_n\in A)}{\{nP(\|H\|>M_n)\}^k}=\frac1{k!}\nu^{(k)}(A)\,.
\end{equation}
We first show the lower bound, {\it i.e.}, the lim inf of the left hand side is at least as much as the right hand side. Fix a set $A$ as described above. Define for $\epsilon>0$
$$A^{-\epsilon}:=\{x\in A:\mbox{ for all }y\in\bbr^d\mbox{ with }\|y-x\|<\epsilon,y\in A\}\,.$$
Clearly,
$$\lim_{\epsilon\downarrow0}\nu^{(k)}(A^{-\epsilon})=\nu^{(k)}(int(A))=\nu^{(k)}(A)\,,$$
where the second equality is true because $A$ is $\nu^{(k)}$-continuous.
Thus, for the lower bound, it suffices to show that for all $\epsilon>0$ so that $A^{-\epsilon}$ is a $\nu^{(k)}$-continuity set (which is true for all but countably many $\epsilon$'s),
\begin{equation}\label{ld.t1.eq2}
\liminf_{n\to\infty}\frac{P(M_n^{-1}S_n\in A)}{\{nP(\|H\|>M_n)\}^k}\ge\frac1{k!}\nu^{(k)}(A^{-\epsilon})\,.
\end{equation}
Fix $\epsilon>0$ so that $A^{-\epsilon}$ is a $\nu^{(k)}$-continuity set. Since we want to show \eqref{ld.t1.eq2}, we can assume without loss of generality that $\nu^{(k)}(A^{-\epsilon})>0$. Fix $n\ge k$ and define for $1\le j_1<\ldots<j_k\le n$
$$C_{j_1\ldots j_k}:=\left\{M_n^{-1}\sum_{u=1}^kX_{nj_u}\in A^{-\epsilon},\biggl\|\sum_{i\in\{1,\ldots,n\}\setminus\{j_1,\ldots,j_k\}}X_{ni}\biggr\|<\epsilon M_n\right\}\,.$$
Though the above definition also depends on $n$, we suppress
that to keep the notation simple. Clearly,
$$P(M_n^{-1}S_n\in A)\ge P\left(\bigcup C_{j_1\ldots j_k}\right)\,,$$
where the union is taken over all subsets of size $k$ of $\{1,\ldots,n\}$,
and
\begin{eqnarray*}
P\left(C_{1,\ldots,k}\right)&=&P\left(M_n^{-1}\sum_{j=1}^{k}X_{nj}\in A^{-\epsilon}\right)P\left(\biggl\|\sum_{i=1}^{n-k}X_{ni}\biggr\|<M_n\epsilon\right)\\
&\sim&P(\|H\|>M_n)^k\nu^{(k)}(A^{-\epsilon})\,,
\end{eqnarray*}
as $n\longrightarrow\infty$, where the equivalence is true because
$M_n^{-1}\sum_{i=1}^{n-k}X_{ni}\prob0$ and by Lemma \ref{ld.l12}.
Thus, for \eqref{ld.t1.eq2}, all that remains to show is
\begin{equation}\label{ld.t2.eq13}
P\left(\bigcup C_{j_1\ldots j_k}\right)\sim\sum P\left(C_{j_1\ldots j_k}\right)\,,
\end{equation}
where the union and the sum are both taken over all subsets of $\{1,\ldots,n\}$. Fix $\eta>0$ so that $(k-1)(1+\eta)<\delta$ and subsets $\{i_1,\ldots,i_k\}$ and $\{j_1,\ldots,j_k\}$ of $\{1,\ldots,n\}$ so that
\begin{equation}\label{ld.subsets}
\#\left(\{i_1,\ldots,i_k\}\cap\{j_1,\ldots,j_k\}\right)=l<k\,.
\end{equation}
Note that,
\begin{eqnarray*}
&&P\left(C_{i_1\ldots i_k}\cap C_{j_1\ldots j_k}\right)\\
&\le&P\biggl(M_n^{-1}\biggl\|\sum_{u=1}^kX_{nj_u}\biggr\|>\delta, M_n^{-1}\biggl\|\sum_{u=1}^kX_{ni_u}\biggr\|>\delta\biggr)\\
&\le&P\biggl(M_n^{-1}\biggl\|\sum_{u=1}^kX_{nj_u}\biggr\|>\delta,M_n^{-1}\biggl\|\sum_{u=1}^kX_{ni_u}\biggr\|>\delta,\\
&&\|X_{nu}\|\le(1+\eta)M_n\mbox{ for }u\in\{i_1,\ldots,i_k\}\cup\{j_1,\ldots,j_k\}\biggr)\\
&&+2kP(L>\eta M_n)P(\|H\|>M_n)\\
&\le&P\left(\|X_{nj}\|>[\delta-(k-1)(1+\eta)]M_n\mbox{ for }1\le j\le2k-l\right)\\
&&+o(P(\|H\|>M_n)^k)\\
&=&O(P(\|H\|>M_n)^{2k-l})\,.
\end{eqnarray*}
Clearly, for fixed $l$, there are at most $O(n^{2k-l})$ pairs of subsets satisfying \eqref{ld.subsets}.
Thus,
\begin{eqnarray*}
\sum P\left(C_{i_1\ldots i_k}\cap C_{j_1\ldots j_k}\right)&=&\sum_{l=0}^{k-1}O(n^{2k-l}P(\|H\|>M_n)^{2k-l})\\
&=&o(n^kP(\|H\|>M_n)^k)\,,
\end{eqnarray*}
where the sum in the left hand side of the first line is taken over all pairs of distinct subsets $\{i_1,\ldots,i_k\}$ and $\{j_1,\ldots,j_k\}$ of $\{1,\ldots,n\}$. This shows \eqref{ld.t2.eq13} and thus completes the proof of the lower bound.

 For the upper bound, choose a sequence $z_n$ satisfying
\begin{eqnarray}\label{ld.t2.eq3}
&\{nP(\|H\|>M_n)\}^{\frac{k+1}{k+2}}\ll nP(\|H\|>z_n)&\\
&\ll\{nP(\|H\|>M_n)\}^{\frac{k}{k+1}}\,,&\mbox{ if }\alpha<2\,,\nonumber\\
\label{ld.t2.eq4}
&nP\left(\|H\|>\frac{M_n}{\log M_n}\right)\ll nP(\|H\|>z_n)&\\
&\ll \min\left(\{nP(\|H\|>M_n)\}^{\frac k{k+1}},nP\left(\|H\|>\frac n{M_n}\right)\right)\,,&\mbox{ if }\alpha>2\,,\nonumber\\
\label{ld.t2.eq6}
&nP\left(\|H\|>\frac{M_n}{\log M_n}\right)\ll nP(\|H\|>z_n)&\\
&\ll \min\left(\{nP(\|H\|>M_n)\}^{\frac k{k+1}},nP\left(\|H\|>\left(\frac n{M_n}\right)^{1+\gamma}\right)\right)\,,&\mbox{ if }\alpha=2\,,\nonumber
\end{eqnarray}
where $\gamma$ is same as that in \eqref{alpha2}. Note that if $u_n$ and $v_n$ are
sequences satisfying $u_n\ll v_n\ll1$, then a
sequence $w_n$ with
$$u_n\ll P(\|H\|>w_n)\ll v_n\,,$$
can be constructed in the following way. Set, for example,
$$
w_n:=U^\leftarrow\left((u_nv_n)^{-1/2}\right)\,,
$$
where
$$
U(\cdot):=1/P(\|H\|>\cdot)\,.
$$
The reader is referred to \cite{resnick:2007} for a definition of $U^\leftarrow(\cdot)$ (page 18), and a proof of the fact that $w_n$ defined as above works (Subsection 2.2.1, page 23-24). Thus, existence of $z_n$ satisfying \eqref{ld.t2.eq3} is immediate from the assumption that $nP(\|H\|>M_n)$ goes to zero as $n\longrightarrow\infty$. In view of \eqref{alphag2}, a sequence satisfying \eqref{ld.t2.eq4} will exist if it can be shown that
\begin{eqnarray}
nP\left(\|H\|>\frac{M_n}{\log M_n}\right)&\ll&\{nP(\|H\|>M_n)\}^\beta\,,\label{ld.cond2}
\end{eqnarray}
when $\alpha>2$, where $\beta=k/(k+1)$. Letting $\epsilon\in(0,\alpha-2)$, $\delta\in(0,\epsilon(1-\beta)/2)$ and $l(x):=x^\alpha P(\|H\|>x)$, note that
\begin{eqnarray}
&&\frac{nP\left(\|H\|>\frac{M_n}{\log M_n}\right)}{\{nP(\|H\|>M_n)\}^\beta}\nonumber\\
&=&n^{1-\beta}M_n^{-\alpha(1-\beta)}(\log M_n)^\alpha l(M_n/\log M_n)l(M_n)^{-\beta}\nonumber\\
&\ll&n^{1-\beta}M_n^{-\alpha(1-\beta)}(\log M_n)^\alpha (M_n/\log M_n)^{\epsilon(1-\beta)/2}M_n^{\epsilon(1-\beta)/2-\delta}\nonumber\\
&=&n^{1-\beta}M_n^{(\epsilon-\alpha)(1-\beta)-\delta}(\log M_n)^c\nonumber\\
&\ll&n^{1-\beta}M_n^{(\epsilon-\alpha)(1-\beta)}\,,\label{new.eq2}
\end{eqnarray}
where $c:=\alpha-\epsilon(1-\beta)/2$. Using the fact that $M_n\gg\sqrt{n}$, which is a consequence of \eqref{alphag2}, it follows that
\begin{eqnarray*}
n^{1-\beta}M_n^{(\epsilon-\alpha)(1-\beta)}&\ll&n^{1-\beta}n^{(\epsilon-\alpha)(1-\beta)/2}\\
&=&n^{(1-\beta)(2-\alpha+\epsilon)/2}\\
&\to&0\mbox{ by choice of }\epsilon\,.
\end{eqnarray*}
This clearly shows \eqref{ld.cond2} when $\alpha>2$.

To establish that a sequence $z_n$ satisfying \eqref{ld.t2.eq6} exists, it suffices to check \eqref{ld.cond2} and that
\begin{equation}\label{ld.t2.eq7}
\frac{M_n}{\log M_n}\gg\left(\frac n{M_n}\right)^{1+\gamma}\,,
\end{equation}
both when $\alpha=2$.
For \eqref{ld.cond2}, let $0<\epsilon<2\gamma/(1+\gamma)<2$, where $\gamma$ is same as that in \eqref{alpha2}. A quick inspection reveals that the arguments leading to \eqref{new.eq2} hold regardless of the values of $\epsilon$ and $\alpha$. Using \eqref{alpha2}, it follows that when $\alpha=2$,
\begin{eqnarray*}
n^{1-\beta}M_n^{(\epsilon-\alpha)(1-\beta)}
&\ll&n^{1-\beta}n^{(\epsilon-2)(1-\beta)(1+\gamma)/2}\\
&\to&0\mbox{ by choice of }\epsilon\,.
\end{eqnarray*}
Thus, \eqref{ld.cond2} holds when $\alpha=2$.
Using \eqref{alpha2} once again, \eqref{ld.t2.eq7} follows.

Write
$$\tilde S_n:=\sum_{j=1}^nX_{nj}\one{(\|X_{nj}\|\le z_n)}\,.$$
Fix $0<\epsilon<\delta-k+1$ and define
$$A^\epsilon:=\{y\in\bbr^d:\|y-x\|<\epsilon\mbox{ for some }x\in A\}\,.$$
Assume that $\epsilon$ is chosen so that $A^\epsilon$ is
also a $\nu^{(k)}$-continuity set. Define the events
\begin{eqnarray*}
D_n&:=&\Biggl\{M_n^{-1}\sum_{u=1}^lX_{nj_u}\in A^\epsilon\mbox{ for at least one tuple }\\
&&1\le j_1<j_2<\ldots<j_l\le n, 1\le l<k\Biggr\}\,,\\
E_n&:=&\Bigg\{M_n^{-1}\sum_{u=1}^kX_{nj_u}\in A^\epsilon\mbox{ for at least one tuple }\\
&&1\le j_1<j_2<\ldots<j_k\le n\Bigg\}\,,\\
F_n&:=&\{\|X_{nj}\|> z_n\mbox{ for at least }(k+1)\mbox{ many }j\mbox{'s}\le n\}\,,\\
G_n&:=&\{\|\tilde S_n\|>\epsilon M_n\}\,.
\end{eqnarray*}
Clearly,
$$P\left(M_n^{-1}S_n\in A\right)\le P(D_n)+P(E_n)+P(F_n)+P(G_n)\,.$$
Also,
\begin{eqnarray*}
&&P(E_n)\\
&\le&\frac{n^k}{k!}P\left(M_n^{-1}\sum_{j=1}^kX_{nj}\in A^\epsilon\right)\\
&\sim&\frac1{k!}\{nP(\|H\|>M_n)\}^k\int\cdots\int\one\left(\sum_{j=1}^kx_{j}\in
A^\epsilon\right)\nu(dx_1)\ldots\nu(dx_k)
\end{eqnarray*}
by Lemma \ref{ld.l12}. By the fact that $A\subset B_\delta^c$ and $\epsilon<\delta-k+1$,
\begin{eqnarray*}
P(D_n)&\le&\sum_{l=1}^{k-1}n^lP\left(\|\sum_{j=1}^lX_{nj}\|>(\delta-\epsilon)M_n\right)\\
&\le&\sum_{l=1}^{k-1} n^llP\left[L>\left\{(\delta-\epsilon)/l-1\right\}M_n\right]P(\|H\|>M_n)\\
&\ll& n^kP(\|H\|>M_n)^k\,,
\end{eqnarray*}
the last inequality following from \eqref{ld.t2.eq2}.
By the choice of $z_n$,
$$P(F_n)\le \{nP(\|H\|>z_n)\}^{k+1}\ll\{nP(\|H\|>M_n)\}^k\,.$$
All that remains is to show that
\begin{equation}\label{ld.t2.eq5}
P(G_n)\ll\{nP(\|H\|>M_n)\}^k\,.
\end{equation}
Recall that $\|\cdot\|$ denotes the $L^2$ norm as defined in \eqref{ld.norm}.
Denoting the coordinates of a $\bbr^d$-valued random variable $Y$ by $Y^{(j)}$ for $1\le j\le d$, note that
\begin{eqnarray*}
P(G_n)&\le&\sum_{j=1}^dP\left(|\tilde S_n^{(j)}|>\epsilon M_n/\sqrt d\right)\,.
\end{eqnarray*}
In view of this, to show \eqref{ld.t2.eq5}, It suffices to prove that for $1\le j\le d$,
\begin{eqnarray}
E S_n^{(j)}&=&o(M_n)\,\label{ld.t2.eq8}\\
P\left(|\tilde S_n^{(j)}-E\tilde S_n^{(j)}|>\theta M_n\right)&=&o\left(\{nP(\|H\|>M_n)\}^k\right)\,,\label{ld.t2.eq9}
\end{eqnarray}
for all $\theta>0$. By the assumption that $H$ has a symmetric law when $\alpha=1$, \eqref{ld.t2.eq8} is trivially true in that case.  We shall show \eqref{ld.t2.eq8} separately for the cases $\alpha<1$ and $\alpha>1$.
 We start with the case $\alpha>1$. Note that for $n$ large enough so that $z_n<M_n$,
\begin{eqnarray*}
|E S_n^{(j)}|&=&n|E[X_{n1}^{(j)}\one(\|X_{n1}\|\le z_n)]|\\
&=&n|E[H^{(j)}\one(\|H\|\le z_n)]|\\
(\mbox{since }EH=0\mbox{ when }\alpha>1)&=&n|E[H^{(j)}\one(\|H\|> z_n)]|\\
&\le&nE[|H^{(j)}|\one(\|H\|> z_n)]\\
&\le&nE[\|H\|\one(\|H\|> z_n)]\\
&=&O(nz_nP(\|H\|>z_n))\\
&=&o(M_n)\,.
\end{eqnarray*}
where the last step follows from the fact that the choice of $z_n$ implies that $z_n\ll M_n$ and that $nP(\|H\|>z_n)\ll1$, which are true, in fact, for all $\alpha$. For the case $\alpha<1$, note that for $n$ large enough,
\begin{eqnarray*}
|E S_n^{(j)}|&=&n|E[X_{n1}^{(j)}\one(\|X_{n1}\|\le z_n)]|\\
&=&n|E[H^{(j)}\one(\|H\|\le z_n)]|\\
&\le&nE[|H^{(j)}|\one(\|H\|\le z_n)]\\
&\le&nE[\|H\|\one(\|H\|\le z_n)]\\
&=&O(nz_nP(\|H\|>z_n))\\
&=&o(M_n)\,.
\end{eqnarray*}
Thus, \eqref{ld.t2.eq8} is established for all $\alpha$.
Note that by Lemma \ref{ld.l1},
$$P\left(|\tilde S_n^{(j)}-E\tilde S_n^{(j)}|>\theta M_n\right)\le K_1\exp\left\{-K_2\frac{M_n}{z_n}\sinh^{-1}K_3\frac{M_nz_n}{\Var(\tilde S_n^{(j)})}\right\}\,,$$
for finite positive constants $K_1, K_2$ and $K_3$. For \eqref{ld.t2.eq9}, all that needs to be shown is
\begin{equation}\label{ld.t2.eq12}
\exp\left\{-K_2\frac{M_n}{z_n}\sinh^{-1}K_3\frac{M_nz_n}{\Var(\tilde S_n^{(j)})}\right\}\ll\{nP(\|H\|>M_n)\}^k\,.
\end{equation}
We shall show this separately for the cases $\alpha<2$ and $\alpha\ge2$. We start with the case $\alpha\ge2$.
For \eqref{ld.t2.eq12}, we claim that it suffices to show that
\begin{eqnarray}
\frac{M_n}{z_n}&\gg&\log M_n\,,\label{ld.t2.eq10}\\
\mbox{and   }M_nz_n&\gg&\Var(\tilde S_n^{(j)})\,.\label{ld.t2.eq11}
\end{eqnarray}
Let $C=2k\alpha$ and notice that
\begin{equation}\label{new3}
M_n^CP(\|H\|>M_n)^k\gg1\gg n^{-k}\,.
\end{equation}
If \eqref{ld.t2.eq10} and \eqref{ld.t2.eq11} are true, it will follow that for large $n$,
$$
\exp\left\{-K_2\frac{M_n}{z_n}\sinh^{-1}K_3\frac{M_nz_n}{\Var(\tilde S_n^{(j)})}\right\}\le M_n^{-C}\,.
$$
In view of \eqref{new3}, this will show \eqref{ld.t2.eq12}.

It follows directly from choice of $z_n$ that \eqref{ld.t2.eq10}
is true. If $\alpha>2$, then
\begin{eqnarray*}
\frac{\Var(\tilde S_n^{(j)})}{M_nz_n}&=&O(n/M_nz_n)\\
&=&o(1)
\end{eqnarray*}
 by choice of $z_n$. If $\alpha=2$, then there is a slowly varying function $m:[0,\infty)\to\bbr$ at $\infty$ so that
\begin{eqnarray*}
\frac{\Var(\tilde S_n^{(j)})}{M_nz_n}&=&O(nm(z_n)/M_nz_n)\\
&=&O\left(n/M_nz_n^{1/(1+\gamma)}\right)\\
&=&o(1)\,.
\end{eqnarray*}
Finally, let us come to the case $\alpha<2$. Note that there is a slowly varying function $m:[0,\infty)\to\bbr$ at $\infty$ (which is possibly different from the one chosen just above), so that
\begin{eqnarray*}
\frac{M_n}{z_n}&\sim&\left(\frac{P(\|H\|>z_n)}{P(\|H\|>M_n)}\right)^{1/\alpha}\frac{m(M_n)}{m(z_n)}\\
&\gg&\left(\frac{P(\|H\|>z_n)}{P(\|H\|>M_n)}\right)^{1/\alpha}\frac{z_n}{M_n}\\
&\gg&\{nP(\|H\|>M_n)\}^{-\frac{1}{\alpha(k+2)}}\frac{z_n}{M_n}\,.
\end{eqnarray*}
This shows that
$$\frac{M_n}{z_n}\gg\{nP(\|H\|>M_n)\}^{-u}$$
for some $u>0$. Also, note that
\begin{eqnarray*}
\Var(\tilde S_n^{(j)})&=&O(nz_n^2P(\|H\|>z_n))\\
&=&o(z_n M_n)\,,
\end{eqnarray*}
the last step following from the facts that $z_n\ll M_n$ and $nP(\|H\|>z_n)\ll1$.
Thus,
$$\frac{M_n}{z_n}\sinh^{-1}K_3\frac{M_nz_n}{\Var(\tilde S_n^{(j)})}\gg\{nP(\|H\|>M_n)\}^{-u}\,,$$
and hence,
\begin{eqnarray*}
\exp\left\{-K_2\frac{M_n}{z_n}\sinh^{-1}K_3\frac{M_nz_n}{\Var(\tilde S_n^{(j)})}\right\}&\ll&\exp\left\{-K_2\{nP(\|H\|>M_n)\}^{-u}\right\}\\
&\ll&\{nP(\|H\|>M_n)\}^k\,.
\end{eqnarray*}
This shows \eqref{ld.t2.eq12} and thus completes the proof.
\end{proof}

Theorem \ref{ld.t2} clearly excludes the boundary cases, {\it i.e.}, it does not give the decay rate of $P(\|S_n\|>kM_n)$ when $k$ is a positive integer.  For stating the results for the boundary case, we need some preliminaries. In view of the assumptions that $E(H)=0$ whenever $\alpha>1$ and that $H$ has a symmetric distribution when $\alpha=1$, by \cite{rvaceva:1962}, it follows that
\begin{equation}\label{ld.eq1}
B_n^{-1}\sum_{j=1}^nH_j\Longrightarrow{\mathcal L}({\mathcal V})\,,
\end{equation}
for some sequence $(B_n)$ going to infinity, and some $(\alpha\wedge2)$-stable random variable $\mathcal V$.
Note that
\begin{eqnarray*}
P\left(S_n\neq\sum_{j=1}^nH_j\right)&\le&P(\|H_j\|>M_n\mbox{ for some }1\le j\le n)\\
&\le&nP(\|H\|>M_n)\\
&\to&0\,.
\end{eqnarray*}
Thus, it follows from \eqref{ld.eq1} that
\begin{equation}\label{ld.eq4}
B_n^{-1}S_n\Longrightarrow{\mathcal L}({\mathcal V})\,.
\end{equation}

The next two results, which are the last two main results of this section, describe the behavior of the large deviation probability for the boundary cases. Specifically, Theorem \ref{ld.t12} gives the decay rate of $P(\|S_n\|>M_n)$ and Theorem \ref{ld.t13} gives the decay rate of $P(\|S_n\|>kM_n)$ for $k\ge2$.

\begin{theorem}\label{ld.t12}$($The boundary case: $k=1)$ In the soft truncation regime, for all closed set $F\subset\mathcal S$,
$$\limsup_{n\to\infty}\frac{P\left(\|S_n\|>M_n,\frac{S_n}{\|S_n\|}\in F\right)}{nP(\|H\|>M_n)}\le\Gamma_1(F)\,,$$
where,
$$\Gamma_1(A):=\int_AP(\langle x,{\mathcal V}\rangle\ge0)\sigma(dx)\,,$$
for $A\subset\mathcal S$, and ${\mathcal V}$ is as in \eqref{ld.eq1}.
If, in addition,
\begin{equation}\label{ld.t12.eq3}
\int_{\mathcal S}P(\langle x,{\mathcal V}\rangle=0)\sigma(dx)=0\,,
\end{equation}
then, as $n\longrightarrow\infty$,
$$\frac{P\left(\|S_n\|>M_n,\frac{S_n}{\|S_n\|}\in\cdot\right)}{nP(\|H\|>M_n)}\weak\Gamma_1(\cdot)$$
weakly on $\mathcal S$.
\end{theorem}

\begin{theorem}\label{ld.t13}$($The boundary case: $k\ge2$) Suppose $k\ge2$ and assume that \eqref{ld.t2.eq2} holds. Then, in the soft truncation regime, 
$$
\limsup_{n\to\infty}\frac{P\left(\|S_n\|>kM_n,\frac{S_n}{\|S_n\|}\in F\right)}{\{nP(\|H\|>M_n)\}^k}\le\Gamma_k(F)\,,
$$
for all closed set $F\subset\mathcal S$, where for all $A\subset\mathcal S$,
$$
\Gamma_k(A):=\frac1{k!}\sum_{s\in A}P(\langle s,{\mathcal V}\rangle\ge0)\sigma(\{s\})^k\,.
$$
If, in addition, for every $s\in\mathcal S$,
\begin{equation}\label{ld.t13.eq1}
\liminf_{t\rightarrow\infty}\frac{P\left(\|H\|>t,\frac H{\|H\|}=s\right)}{P(\|H\|>t)}\ge\sigma(\{s\})
\end{equation}
and
\begin{equation}\label{ld.t13.eq3}
P(\langle s,{\mathcal V}\rangle=0)\sigma(\{s\})=0\,,
\end{equation}
then,
$$\frac{P\left(\|S_n\|>kM_n,\frac{S_n}{\|S_n\|}\in\cdot\right)}{\{nP(\|H\|>M_n)\}^k}\weak\Gamma_k(\cdot)\,,$$
weakly on $\mathcal S$.
\end{theorem}

Before getting into the proof, let us try to understand the need for the assumption \eqref{ld.t13.eq1} when $k\ge2$. Continuing on the note of the heuristic arguments after the statement of Theorem \ref{ld.t2}, one would expect that for $\|S_n\|$ to be at least as large as $kM_n$, it would be ``necessary'' for the sum of some $k$ many of $X_{n1}\ldots,X_{nn}$ to have norm at least $kM_n$. For that to happen when $k\ge2$, one would need that the directions of each of those $k$ summands to be the same. Given any direction $s$, this is possible only when the spectral measure admits an atom at $\{s\}$, and \eqref{ld.t13.eq1} holds. This clearly isn't true for $k=1$, in which case, the sum of $k$ random variables is actually the random variable itself, and the norm of a particular $X_{nj}$ being at least as large as $M_n$ is equivalent to $\|H_j\|\ge M_n$.

It is easy to see that for all $k\ge1$, $\Gamma_k({\mathcal S})\le\sigma({\mathcal S})=1$, which in particular implies that $\Gamma_k$ is a finite measure. However, $\Gamma_k$ might be the null measure, and if that is the case, the statements  of Theorems \ref{ld.t12} and \ref{ld.t13} just mean that $P(\|S_n\|>kM_n)$ decays faster than $\{nP(\|H\|>M_n)\}^k$. For the proofs, we shall need the following lemma, which in fact, proves the first parts of both theorems.

\begin{lemma}\label{ld.l2} Suppose $k\ge1$ and assume that  \eqref{ld.t2.eq2} holds. Then, as $n\longrightarrow\infty$,
$$
\limsup_{n\to\infty}\frac{P\left(\|S_n\|>kM_n,\frac{S_n}{\|S_n\|}\in F\right)}{\{nP(\|H\|>M_n)\}^k}\le\Gamma_k(F)
\,,
$$
for all closed set $F\subset\mathcal S$.
\end{lemma}

\begin{proof} It is easy to see that for all $k\ge1$ and $A\subset\mathcal S$,
$$\Gamma_k(A)=\frac1{k!}\int_{{\mathcal S}}\ldots\int_{{\mathcal S}}\one\left(\left\|\sum_{j=1}^kx_j\right\|=k,\frac{\sum_{j=1}^kx_j}{\|\sum_{j=1}^kx_j\|}\in A\right)P\left(\sum_{j=1}^k\langle x_j,{\mathcal V}\rangle\ge0\right)$$
$$
\sigma(dx_1)\ldots\sigma(dx_k)\,.$$
Fix $k\ge1$ and a closed set $F\subset\mathcal S$. Let $0<\eta<1$ and define
$$
E_n:=
\Biggl\{\left\|\sum_{u=1}^kX_{nj_u}\right\|>(k-\eta)M_n\mbox{ for at least one tuple }
$$
$$
\,\,\,\,\,\,\,\,
1\le j_1<j_2<\ldots<j_k\le n\Biggr\}\,.
$$
By similar arguments as in the proof of Theorem \ref{ld.t2}, it follows that
$$P\left(\{\|S_n\|>kM_n\}\cap E_n^c\right)=o(\{nP(\|H\|>M_n)\}^k)$$
as $n\longrightarrow\infty$. Thus, for the upper bound, it suffices to show that
$$\limsup_{\eta\downarrow0}\limsup_{n\rightarrow\infty}\frac{P\left(\left\{\|S_n\|>kM_n,\frac{S_n}{\|S_n\|}\in F\right\}\cap E_n\right)}{\{nP(\|H\|>M_n)\}^k}$$
$$\le\frac1{k!}\int_{\mathcal S}\ldots\int_{\mathcal S}\one\left(\|\sum_{j=1}^kx_j\|=k,\frac{\sum_{j=1}^kx_j}{\|\sum_{j=1}^kx_j\|}\in F\right)P\left(\sum_{j=1}^k\langle x_j,{\mathcal V}\rangle\ge0\right)$$
$$\sigma(dx_1)\ldots\sigma(dx_k)\,.$$
and for that it suffices to show
$$\limsup_{\eta\downarrow0}\limsup_{n\rightarrow\infty}\frac{P\left(\|S_n\|>kM_n,\frac{S_n}{\|S_n\|}\in F,\|\sum_{j=1}^kX_{nj}\|>(k-\eta)M_n\right)}{P(\|H\|>M_n)^k}$$
$$\le\int_{\mathcal S}\ldots\int_{\mathcal
S}\one\left(\|\sum_{j=1}^kx_j\|=k,\frac{\sum_{j=1}^kx_j}{\|\sum_{j=1}^kx_j\|}\in
F\right)P\left(\sum_{j=1}^k\langle
x_j,{\mathcal V}\rangle\ge0\right)$$
\begin{equation}\label{ld.l13.eq1}
\sigma(dx_1)\ldots\sigma(dx_k)\,.
\end{equation}
Fix a sequence $\epsilon_n$ satisfying $M_n^{-1}\ll\epsilon_n\ll M_n^{-1}B_n$, which is possible because $B_n$ goes to infinity, where $B_n$ is as in \eqref{ld.eq1}. Also  $B_n=O(b_n)=o(M_n)$, where $b_n$ is as defined in \eqref{ld.eq5}, thus showing that $\epsilon_n$ goes to zero as $n$ goes to infinity. Set
$$F^\eta:=\{x\in{\mathcal S}:\|x-s\|\le\eta\mbox{ for some }s\in F\}\,.$$
Define the events
\begin{eqnarray*}
U_n&:=
&\left\{\|\sum_{j=1}^kX_{nj}\|>(k-\eta)M_n,\frac{\sum_{j=1}^kX_{nj}}{\|\sum_{j=1}^kX_{nj}\|}\in F^\eta,\right.\\
&&\left.\left\langle\frac{\sum_{j=1}^kX_{nj}}{\|\sum_{j=1}^kX_{nj}\|},B_n^{-1}\sum_{j=k+1}^nX_{nj}\right\rangle\ge-\eta\right\}\,,\\
V_n&:=&\left\{k-\eta<M_n^{-1}{\|\sum_{j=1}^kX_{nj}\|}\le\sqrt{k^2+\epsilon_n},\|S_n\|>kM_n,\right.\\
&&\left.\left\langle\frac{\sum_{j=1}^kX_{nj}}{\|\sum_{j=1}^kX_{nj}\|},B_n^{-1}\sum_{j=k+1}^nX_{nj}\right\rangle<-\eta\right\}\,,\\
W_n&:=
&\Biggl\{\|\sum_{j=1}^kX_{nj}\|>(k-\eta)M_n,\|S_n\|>M_n,\frac{\sum_{j=1}^kX_{nj}}{\|\sum_{j=1}^kX_{nj}\|}\notin F^\eta,\\
&&\frac{S_n}{\|S_n\|}\in F\Biggr\}\,,\\
Y_n&:=
&\left\{\|\sum_{j=1}^kX_{nj}\|>(k-\eta)M_n,\min_{1\le j\le k}\|X_{nj}\|<\frac{1-\eta}2M_n\right\}\,,\\
Z_n&:=&\left\{\min_{1\le j\le
k}\|X_{nj}\|\ge\frac{1-\eta}2M_n,\|\sum_{j=1}^kX_{nj}\|>\sqrt{k^2+\epsilon_n}M_n\right\}\,.
\end{eqnarray*}
Note that
$$\left\{\|S_n\|>kM_n,\frac{S_n}{\|S_n\|}\in F,\|\sum_{j=1}^kX_{nj}\|>(k-\eta)M_n\right\}\subset U_n\cup V_n\cup W_n\cup Y_n\cup Z_n\,.$$
Let $k-1<r<k-\eta$ be such that
$$\nu^{(k)}\left(\{x\in\bbr^d:\|x\|=r\}\right)=0\,.$$
For $n\ge1$, let $P_n(\cdot)$ and $\tilde\nu^{(k)}$ denote the restrictions of\\ $P\left(M_n^{-1}\sum_{j=1}^kX_{nj}\in\cdot\right)$ and $\nu^{(k)}$ respectively to $\bbr^d\setminus B_r$, {\it i.e.}, for $A\subset\bbr^d$,
\begin{eqnarray*}
P_n(A)&:=&P\left(M_n^{-1}\sum_{j=1}^kX_{nj}\in A\cap B_r^c\right)\,,\\
\tilde\nu^{(k)}(A)&:=&\nu^{(k)}\left( A\cap B_r^c\right)\,.
\end{eqnarray*}
Then, by Lemma \ref{ld.l12}, it follows that
$$\frac{P_n(\cdot)}{P(\|H\|>M_n)^k}\weak\tilde\nu^{(k)}(\cdot)\,.$$
By \eqref{ld.eq4}, it follows that
$$\frac{P_n(dx)}{P(\|H\|>M_n)^k}P\left(B_n^{-1}\sum_{j=k+1}^nX_{nj}\in dy\right)\weak\tilde\nu^{(k)}(dx)P({\mathcal V}\in dy)$$
on $\bbr^d\times\bbr^d$. Note that
\begin{eqnarray*}
&&P(U_n)\\
&=&\int_{\bbr^d}\int_{\bbr^d}\one\left(\|x\|>k-\eta,\frac x{\|x\|}\in
F^\eta\right)\one(\langle
x,y\rangle\ge-\eta)P_n(dx)\\
&&P\left(B_n^{-1}\sum_{j=k+1}^nX_{nj}\in dy\right)\,.
\end{eqnarray*}
Since $F^\eta$ is a closed set,
\begin{eqnarray*}
&&\limsup_{n\rightarrow\infty}\frac{P(U_n)}{P(\|H\|>M_n)^k}\\
&\le&\int\one\left(\|x\|\ge k-\eta,\frac x{\|x\|}\in
F^\eta\right)P(\langle x,{\mathcal V}\rangle\ge-\eta)\tilde\nu^{(k)}(dx)\\
&=&\int\one\left(\|x\|\ge k-\eta,\frac x{\|x\|}\in
F^\eta\right)P(\langle x,{\mathcal V}\rangle\ge-\eta)\nu^{(k)}(dx)\,.
\end{eqnarray*}
Letting $\eta\downarrow0$, we get using the fact that $F$ is a closed set,
\begin{eqnarray*}
&&\limsup_{\eta\downarrow0}\limsup_{n\rightarrow\infty}\frac{P(U_n)}{P(\|H\|>M_n)^k}\\
&\le&\int_{\bbr^d}\one\left(\|x\|\ge k,\frac x{\|x\|}\in F\right)P(\langle x,{\mathcal V}\rangle\ge0)\nu^{(k)}(dx)\\
&=&\int_{\bbr^d}\ldots\int_{\bbr^d}\one\left(\|\sum_{j=1}^kx_j\|\ge k,\frac{\sum_{j=1}^kx_j}{\|\sum_{j=1}^kx_j\|}\in F\right)P\left(\sum_{j=1}^k\langle x_j,{\mathcal V}\rangle\ge0\right)\\
&&\nu(dx_1)\ldots\nu(dx_k)\\
&=&\int_{\mathcal S}\ldots\int_{\mathcal
S}\one\left(\left\|\sum_{j=1}^kx_j\right\|=k,\frac{\sum_{j=1}^kx_j}{\|\sum_{j=1}^kx_j\|}\in
F\right)P\left(\sum_{j=1}^k\langle
x_j,{\mathcal V}\rangle\ge0\right)\\
&&\sigma(dx_1)\ldots\sigma(dx_k)\,,
\end{eqnarray*}
the last equality being true because $\nu(B_1^c)=0$ and the restriction of $\nu$ to $\mathcal S$ is $\sigma$.
Thus, in order to show (\ref{ld.l13.eq1}), all that remains is to
prove that
$$P(V_n)+P(W_n)+P(Y_n)+P(Z_n)\ll P(\|H\|>M_n)^k\,.$$
Note that on the set $V_n$,
\begin{eqnarray*}
k^2M_n^2&<&\|S_n\|^2\\
&=&\left\|\sum_{j=1}^k X_{nj}\right\|^2+\left\|\sum_{j=k+1}^n X_{nj}\right\|^2+2\left\langle\sum_{j=1}^k X_{nj},\sum_{j=k+1}^n X_{nj}\right\rangle\\
&\le&(k^2+\epsilon_n)M_n^2+\left\|\sum_{j=k+1}^n X_{nj}\right\|^2-2B_n\eta\left\|\sum_{j=1}^k X_{nj}\right\|\\
&\le&(k^2+\epsilon_n)M_n^2+\left\|\sum_{j=k+1}^n X_{nj}\right\|^2-2\eta(k-\eta)B_nM_n\,,
\end{eqnarray*}
and hence,
\begin{eqnarray*}
&&P(V_n)\\
&\le&P\left(\left\|\sum_{j=1}^kX_{nj}\right\|\ge(k-\eta)M_n\right)\\
&&\times P\left(\left\|\sum_{j=k+1}^nX_{nj}\right\|^2>2\eta(k-\eta)B_nM_n-\epsilon_nM_n^2\right)\\
&\ll&P(\|H\|>M_n)^k\,,
\end{eqnarray*}
the last step following from the fact that  by the choice of $\epsilon_n$, $\epsilon_nM_n^2+B_n^2=o(B_nM_n)$ showing that $2\eta(k-\eta)B_nM_n-\epsilon_nM_n^2$ is much larger than $B_n^2$ which is the growth rate of $\left\|\sum_{j=k+1}^nX_{nj}\right\|^2$. Since for any $u,v\in\bbr^d$,
\begin{eqnarray*}
\left\|\frac{u+v}{\|u+v\|}-\frac u{\|u\|}\right\|
&\le&\left\|\frac{u+v}{\|u+v\|}-\frac u{\|u+v\|}\right\|+\left\|\frac u{\|u+v\|}-\frac u{\|u\|}\right\|\\
&=&\frac{\|v\|}{\|u+v\|}+\left|\frac{\|u+v\|-\|u\|}{\|u+v\|}\right|\\
&\le&2\frac{\|v\|}{\|u+v\|}\,,
\end{eqnarray*}
it follows that
\begin{eqnarray*}
P(W_n)&\le&P\left(\left\|\sum_{j=1}^kX_{nj}\right\|\ge(k-\eta)M_n\right)P\left(\left\|\sum_{j=k+1}^nX_{nj}\right\|>\frac\eta2M_n\right)\\
&\ll&P(\|H\|>M_n)^k\,.
\end{eqnarray*}
Clearly,
\begin{eqnarray*}
P(Y_n)&\le&\sum_{j=1}^kP\left(\|X_{nj}\|>\frac{2k-1-\eta}{2(k-1)}M_n\right)\\
&\le&kP(\|H\|>M_n)P\left(L>\frac{1-\eta}{2(k-1)}M_n\right)\\
&\ll&P(\|H\|>M_n)^k\,,
\end{eqnarray*}
the last step following by \eqref{ld.t2.eq2}.
Finally,
\begin{eqnarray*}
P(Z_n)&\le&kP\left(\|H\|>\frac{1-\eta}2M_n\right)^kP\left(L>\left(\frac{\sqrt{k^2+\epsilon_n}}k-1\right)M_n\right)\\
&\ll&P(\|H\|>M_n)^k\,,
\end{eqnarray*}
the last step being true because by the choice of $\epsilon_n$, it follows that
\begin{eqnarray*}
1&\ll&\epsilon_nM_n\\
&=&O\left(\left(\frac{\sqrt{k^2+\epsilon_n}}k-1\right)M_n\right)\,.
\end{eqnarray*}
This completes the proof.
\end{proof}

\begin{proof}[Proof of Theorem \ref{ld.t12}] In view of Lemma \ref{ld.l2}, it suffices to show that
\begin{equation}\label{ld.t12.eq1}
\liminf_{n\to\infty}\frac{P(\|S_n\|>M_n)}{nP(\|H\|>M_n)}\ge\Gamma_1({\mathcal S})\,.
\end{equation}
We assume without loss of generality that $\Gamma_1({\mathcal S})>0$.
For $1\le j\le n$, define
$$C_j:=\left\{\|X_{nj}\|\ge M_n,\sum_{1\le i\le n, i\neq j}\langle X_{ni},X_{nj}\rangle>0\right\}\,.$$
Note that
\begin{equation}\label{ld.t12.eq2}
P(\|S_n\|>M_n)\ge P\left(\bigcup_{j=1}^n C_j\right)\,,
\end{equation}
and that
\begin{eqnarray}
P(C_j)&=&\int_{\mathcal S}\int_{\bbr^d}\one(\langle x,y\rangle>0)P\left(\|X_{n1}\|\ge M_n,\frac{X_{n1}}{\|X_{n1}\|}\in dx\right)\nonumber\\
&&\,\,\,\,\,\,\,P\left(B_n^{-1}\sum_{j=2}^nX_{nj}\in dy\right)\nonumber\\
&=&\int_{\mathcal S}\int_{\bbr^d}\one(\langle x,y\rangle>0)P\left(\|H\|\ge M_n,\frac{H}{\|H\|}\in dx\right)\nonumber\\
&&\,\,\,\,\,\,\,P\left(B_n^{-1}\sum_{j=2}^nX_{nj}\in dy\right)\nonumber
\end{eqnarray}
By  \eqref{ld.spectral} and \eqref{ld.eq4}, it follows that
\begin{eqnarray}
\liminf_{n\to\infty}\frac{P(C_j)}{P(\|H\|>M_n)}&\ge&\int_{\mathcal S}\int_{\bbr^d}\one(\langle x,y\rangle>0)\sigma(dx)P\left({\mathcal V}\in dy\right)\nonumber\\
&=&\Gamma_1({\mathcal S})\,,\label{ld.t12.eq4}
\end{eqnarray}
the equality in the last line following from \eqref{ld.t12.eq3}. In view of \eqref{ld.t12.eq2} and \eqref{ld.t12.eq4}, all that needs to be shown is that
$$n^2P(C_1\cap C_2)=o(nP(\|H\|>M_n))\,,$$
but that follows from similar arguments as in the proof of Theorem \ref{ld.t2}. This completes the proof.
\end{proof}

\begin{proof}[Proof of Theorem \ref{ld.t13}] In view of Lemma \ref{ld.l2}, it suffices to show that if \eqref{ld.t13.eq1} and \eqref{ld.t13.eq3} hold, then for $k\ge2$ and $s_1,\ldots,s_r\in\mathcal S$,
\begin{equation}\label{ld.t13.eq2}
\liminf_{n\to\infty}\frac{P(\|S_n\|>M_n)}{\{nP(\|H\|>M_n)\}^k}\ge\frac1{k!}\sum_{i=1}^rP(\langle s_i,{\mathcal V}\rangle\ge0)\sigma(\{s_i\})^k\,.
\end{equation}
Denote for $1\le j_1<\ldots<j_k\le n$,
$$C_{j_1\ldots j_k}:=\bigcup_{i=1}^r\left\{\|H_{j_u}\|\ge M_n,\frac{H_{j_u}}{\|H_{j_u}\|}=s_i\mbox{ for }1\le u\le k,\sum_{v\neq j_1,\ldots,j_k}\langle s_i,X_{nv}\rangle>0\right\}\,.$$
Note that,
$$
P(\|S_n\|>k M_n)\ge
P\left(\bigcup C_{j_1\ldots j_k}\right)\,,
$$
where the union is taken over all tuples $1\le j_1<\ldots<j_k\le n$. It follows by \eqref{ld.t13.eq1} and \eqref{ld.t13.eq3} that for any $1\le j_1<\ldots<j_k\le n$ and $1\le i\le r$,
$$
\liminf_{n\to\infty}\frac{P\left(\|H_{j_u}\|\ge M_n,\frac{H_{j_u}}{\|H_{j_u}\|}=s_i\mbox{ for }1\le u\le k,\sum_{v\neq j_1,\ldots,j_k}\langle s_i,X_{nv}\rangle>0\right)}{P(\|H\|>M_n)^k}
$$
$$
\ge\sigma(\{s_i\})^kP(\langle s_i,{\mathcal V}\rangle\ge0)\,,
$$
and hence for $1\le j_1<\ldots<j_k\le n$,
$$
\liminf_{n\to\infty}\frac{P\left(C_{j_1\ldots j_k}\right)}{P(\|H\|>M_n)^k}\ge\sum_{i=1}^r\sigma(\{s_i\})^kP(\langle s_i,{\mathcal V}\rangle\ge0)\,.
$$
Thus, in order to show \eqref{ld.t13.eq2}, it suffices to prove that as $n\longrightarrow\infty$,
$$
P\left(\bigcup C_{j_1\ldots j_k}\right)\sim\sum P\left(C_{j_1\ldots j_k}\right)\,,
$$
where the sum and the union are taken over all tuples $1\le j_1<\ldots<j_k\le n$. That follows from similar arguments leading to the proof of \eqref{ld.t2.eq13}. This completes the proof.
\end{proof}

\section{Large deviations: the hard truncation regime}\label{sec:ld.ht}The setup for this section is similar to that in Section \ref{sec:ld.st}, except that now we are in the hard truncation regime. That is,  $H$ is a $\bbr^d$-valued random variable such that \eqref{ld.defn} holds.
If $\alpha=1$, then $H$ is assumed to have a symmetric law and if $\alpha>1$, then $EH=0$. 

For this section, we assume that $M_n$ goes to $\infty$ slowly enough so that
\begin{equation}\label{hard}
\lim_{n\to\infty}nP(\|H\|>M_n)=\infty\,,
\end{equation}
an equivalent formulation of which is
\begin{equation}\label{ht}
1\ll M_n\ll a_n\,,
\end{equation}
where $a_n$ is same as the one in \eqref{ld.defn}.
Moreover, we assume that
$$
E\|H\|^2<\infty\mbox{ if }\alpha=2\,.
$$
We further assume that $Ee^{\epsilon L}<\infty$ for some $\epsilon>0$.

A sequence of random variables $Z_n$ follows the Large Deviations Principle (LDP) with speed $c_n$ and rate function $I$ if for any Borel set
 $A$,
$$-\inf_{x\in int(A)}I(x)\le\liminf_{n\rightarrow\infty}\frac1{c_n}\log P(Z_n\in A)$$
$$\le\limsup_{n\rightarrow\infty}\frac1{c_n}\log P(Z_n\in A)\le-\inf_{x\in cl(A)}I(x)\,,$$
where $int(\cdot)$ and $cl(\cdot)$ denote the interior and the closure of a set respectively, as before.

The first result of this section is an analogue of Cram\'er's
Theorem (Theorem 2.2.3, page 27 in \cite{dembo:zeitouni:1998})
because of the following reason. Recall that Cram\'er's Theorem
gives the LDP for $n^{-1}\sum_{i=1}^nZ_i$ where $Z_1,Z_2,\ldots$
are i.i.d. random variables with finite exponential moments. Note
that the normalizing constant is $n$, the rate at which
$E\sum_{i=1}^n\|Z_i\|$ grows. The following result gives the LDP
for the sequence $S_n/\{nM_nP(\|H\|>M_n)\}$. By Karamata's
Theorem, it is easy to see that if $\alpha<1$,
$$E\sum_{i=1}^n\biggl\|H_i\one\left(\|H_i\|\le  M_n\right)+\frac{H_i}{\|H_i\|}(M_n+L_i)\one\left(\|H_i\|>M_n\right)\biggr\|$$
grows like $nM_nP(\|H\|>M_n)$ up to a constant, and hence we
consider this to be an analogue of Cram\'er's Theorem, at least for that case. This result, however, is valid for $\alpha<2$.

\begin{theorem}[Large Deviations ($\alpha<2)$]\label{ld.t8}In the hard truncation regime, the random variable $$S_n/\{nM_nP(\|H\|>M_n)\}$$ follows LDP with speed $nP(\|H\|>M_n)$ and rate function $\Lambda^*$, which is the Fenchel-Legendre transform (refer to Definition 2.2.2, page 26 in \cite{dembo:zeitouni:1998}) of the function $\Lambda$ given by
$$\Lambda(\lambda):=\left\{\begin{array}{ll}\int_{\bbr^d}\left(e^{\langle\lambda,x\rangle}-1\right)\nu(dx),&0<\alpha<1\,,\\\int_{\bbr^d}\left(e^{\langle\lambda,x\rangle}-1-\langle\lambda,x\rangle\right)\nu(dx),&\alpha=1\,,\\\int_{\bbr^d}\left(e^{\langle\lambda,x\rangle}-1-\langle\lambda,x\rangle\right)\nu(dx)-\frac1{\alpha-1}\int_{\mathcal S}\langle\lambda,s\rangle\sigma(ds),&1<\alpha<2\,,\end{array}\right.$$
where $\mathcal S$ and the measures $\sigma$ and $\nu$ are as defined in \eqref{defS}, \eqref{defsigma} and \eqref{ld.t2.eq1} respectively.
\end{theorem}

\begin{proof}
We start by showing that $\Lambda(\lambda)$ is well defined, that is, the integrals defining it exist. We shall show this for the case $0<\alpha<1$, the rest are similar. To that end, notice that for $A\subset\bbr^d$,
$$
\nu(A)=\int_{\mathcal S}\int_{(0,1]}\one(rs\in A)\gamma(dr)\sigma(ds)\,,
$$
where $\gamma$ is the measure on $(0,1]$ defined by
$$
\gamma(dr):=\alpha r^{-\alpha-1}dr+\delta_1(dr)\,,
$$
and $\delta_1$ denotes the measure that gives a point mass to $1$.
Thus,
\begin{eqnarray*}
\int\left|e^{\langle\lambda,x\rangle}-1\right|\nu(dx)&=&\int_{\mathcal S}\int_{(0,1]}\left|e^{r\langle\lambda,s\rangle}-1\right|\gamma(dr)\sigma(ds)\\
&\le&\|\lambda\|e^{\|\lambda\|}\int_{(0,1]}r\gamma(dr)<\infty
\end{eqnarray*}
when $0<\alpha<1$. Thus, $\Lambda(\lambda)$ is well defined in this case. Furthermore, a similar estimate will show that the partial derivatives of the integrand (in the integral defining $\Lambda(\lambda)$) with respect to $\lambda$ are integrable with respect to $\nu$. Due to sufficient smoothness of the integrand, it follows that  $\Lambda(\cdot)$ is differentiable.

Define
\begin{eqnarray*}
X_n&:=&H\one(\|H\|\le M_n)+\frac
H{\|H\|}(M_n+L)\one(\|H\|>M_n)\,.
\end{eqnarray*}
Since $\Lambda$ is  a differentiable function, using the G\"artner-Ellis theorem (Theorem
2.3.6 (page 44) in \cite{dembo:zeitouni:1998}), it suffices to show that for
all $\lambda\in{\mathbb R}^d$,
\begin{equation}\label{ld.t8.e1}
\lim_{n\rightarrow\infty}\frac1{P(\|H\|>M_n)}\log
E\exp(\langle\lambda,M_n^{-1}X_n\rangle)=\Lambda(\lambda)\,.
\end{equation}
This will be shown separately for the cases $\alpha<1$, $\alpha=1$ and $\alpha>1$.
For the first case, note that
\begin{eqnarray*}
E\exp(\langle\lambda,M_n^{-1}X_n\rangle)
&=&1+\int_{{\mathbb
R}^d\setminus\{0\}}\left(e^{\langle\lambda,x\rangle}-1\right)P(M_n^{-1}X_n\in
dx)\,.
\end{eqnarray*}
By Lemma \ref{ld.l11} and the fact that $\nu$ charges only
$\{x:0<\|x\|\le1\}$, for all $0<\epsilon<1$, it follows that
\begin{eqnarray}
&&\int_{\{\epsilon\le\|x\|\le3\}}\left(e^{\langle\lambda,x\rangle}-1\right)P(M_n^{-1}X_n\in
dx)\nonumber\\
\label{ld.t8.e4}
&\sim&
P(\|H\|>M_n)\int_{\{\|x\|\ge\epsilon\}}\left(e^{\langle\lambda,x\rangle}-1\right)\nu(dx)\,.
\end{eqnarray}
For $\alpha<1$ , $e^{\langle\lambda,x\rangle}-1$ is
$\nu$-integrable and hence,
$$\lim_{\epsilon\downarrow0}\int_{\{\|x\|\ge\epsilon\}}\left(e^{\langle\lambda,x\rangle}-1\right)\nu(dx)=\int\left(e^{\langle\lambda,x\rangle}-1\right)\nu(dx)\,.$$
Also,
\begin{eqnarray*}
&&\frac1{P(\|H\|>M_n)}\int_{\{\|x\|>3\}}\left|e^{\langle\lambda,x\rangle}-1\right|P(M_n^{-1}X_n\in dx)\\
&\le&\frac1{P(\|H\|>M_n)}E\left[\exp\left(\langle\lambda,M_n^{-1}X_n\rangle\right)\one(\|M_n^{-1}X_n\|>3)\right]\\
&&+P(L>2M_n)\,.
\end{eqnarray*}
By the Cauchy-Schwartz inequality,
\begin{eqnarray*}
&&\frac1{P(\|H\|>M_n)}E\left[\exp\left(\langle\lambda,M_n^{-1}X_n\rangle\right)\one(\|M_n^{-1}X_n\|>3)\right]\\
&\le&\left[E\exp\left(2M_n^{-1}\|\lambda\|\|X_n\|\right)\right]^{1/2}\frac{P(\|X_n\|>3M_n)^{1/2}}{P(\|H\|>M_n)}\,.
\end{eqnarray*}
Choose $n$ large
enough so that $M_n>\max(1,2\|\lambda\|/\epsilon)$ where $\epsilon$ is
such that $Ee^{\epsilon L}<\infty$. Also, observe that
$$M_n^{-1}\|X_n\|\le(2+M_n^{-1}L)\,.$$
Thus,
$$E\exp\left(2M_n^{-1}\|\lambda\|\|X_n\|\right)\le\exp(4\|\lambda\|)Ee^{\epsilon L}<\infty\,,$$
while,
$$\frac{P(\|X_n\|>3M_n)^{1/2}}{P(\|H\|>M_n)}=\frac{P(L>2M_n)^{1/2}}{P(\|H\|>M_n)^{1/2}}\le\frac{e^{-\epsilon M_n}}{P(\|H\|>M_n)^{1/2}}Ee^{\epsilon L/2}$$
$$\longrightarrow0\,.$$
This shows
\begin{equation}\label{ld.t8.e2}
\lim_{n\longrightarrow\infty}\frac1{P(\|H\|>M_n)}\int_{\{\|x\|>3\}}\left|e^{\langle\lambda,x\rangle}-1\right|P(M_n^{-1}X_n\in
dx)=0\,.
\end{equation}
By Karamata's theorem and the fact that
$e^{\langle\lambda,x\rangle}=1+O(\|x\|)$, one can show that there is $C<\infty$ so that,
$$
\limsup_{n\to\infty}\frac1{P(\|H\|>M_n)}\int_{\{\|x\|<\epsilon\}}\left|e^{\langle\lambda,x\rangle}-1\right|P(M_n^{-1}X_n\in
dx)\le C\epsilon^{1-\alpha}\,,
$$
thus proving that
\begin{equation}\label{ld.t8.e3}
\lim_{\epsilon\downarrow0}\limsup_{n\longrightarrow\infty}
\frac1{P(\|H\|>M_n)}\int_{\{\|x\|<\epsilon\}}\left|e^{\langle\lambda,x\rangle}-1\right|P(M_n^{-1}X_n\in
dx)=0\,.
\end{equation}
Clearly, (\ref{ld.t8.e4}), (\ref{ld.t8.e2}) and (\ref{ld.t8.e3})
show (\ref{ld.t8.e1}) and hence complete the proof for the case $\alpha<1$.

For the case $\alpha=1$, by the fact that when $\alpha=1$, $H$ (and hence $X_n$) has a symmetric distribution it follows that
\begin{eqnarray*}
E\exp(\langle\lambda,M_n^{-1}X_n\rangle)
&=&1+\int_{{\mathbb
R}^d\setminus\{0\}}\left(e^{\langle\lambda,x\rangle}-1-\langle\lambda,x\rangle\right)P(M_n^{-1}X_n\in
dx)\,.
\end{eqnarray*}
Note that $\alpha=1$ implies that $e^{\langle\lambda,x\rangle}-1-\langle\lambda,x\rangle$ is $\nu$-integrable.
By arguments similar to those for the case $\alpha<1$, it follows that as $n\longrightarrow\infty$,
\begin{eqnarray}
&&\int_{{\mathbb
R}^d\setminus\{0\}}\left(e^{\langle\lambda,x\rangle}-1-\langle\lambda,x\rangle\right)P(M_n^{-1}X_n\in
dx)\nonumber\\
\label{ld.t8.e7}
&\sim& P(\|H\|>M_n)\int\left(e^{\langle\lambda,x\rangle}-1-\langle\lambda,x\rangle\right)\nu(dx)\,.
\end{eqnarray}
This completes the proof for the case $\alpha=1$.

For the case $1<\alpha<2$, note that
$$
E\exp(\langle\lambda,M_n^{-1}X_n\rangle)
$$
$$
=1+\int_{{\mathbb
R}^d\setminus\{0\}}\left(e^{\langle\lambda,x\rangle}-1-\langle\lambda,x\rangle\right)P(M_n^{-1}X_n\in
dx)+\int\langle\lambda,x\rangle P(M_n^{-1}X_n\in dx)\,.
$$
For this case also, $e^{\langle\lambda,x\rangle}-1-\langle\lambda,x\rangle$ is clearly $\nu$-integrable, and similar arguments as those for the case $\alpha<1$ show \eqref{ld.t8.e7}.
Thus, all that needs to be shown is as $n\longrightarrow\infty$,
\begin{equation}\label{ld.t8.e6}
\int\langle\lambda,x\rangle P(M_n^{-1}X_n\in dx)\sim-\frac1{\alpha-1}P(\|H\|>M_n)\int_{\mathcal S}\langle\lambda,s\rangle\sigma(ds)\,.
\end{equation}
For this, note that
\begin{eqnarray*}
&&\int\langle\lambda,x\rangle P(M_n^{-1}X_n\in dx)\\
&=&\int_{\{\|x\|\le M_n\}}\langle\lambda,x\rangle P(M_n^{-1}H\in dx)\\
&&\,\,\,\,+\left(1+M_n^{-1}E(L)\right)\int_{\mathcal S}\langle\lambda,s\rangle P\left(\frac H{\|H\|}\in ds,\|H\|>M_n\right)\\
&=:&I_1+I_2\,.
\end{eqnarray*}
By the assumption that $EH=0$, it follows that
\begin{eqnarray*}
I_1&=&-\int_{\{\|x\|>M_n\}}\langle\lambda,x\rangle P(M_n^{-1}H\in dx)\\
&=&-M_n^{-1}\int_{M_n}^\infty\int_{\mathcal S}\langle\lambda,s\rangle rP\left(\frac H{\|H\|}\in ds,\|H\|\in dr\right)\\
&\sim&-P(\|H\|>M_n)\frac\alpha{\alpha-1}\int_{\mathcal S}\langle\lambda,s\rangle\sigma(ds)\,,
\end{eqnarray*}
the equivalence in the last line following by a result similar to Lemma 2.1 in \cite{chakrabarty:samorodnitsky:2009}. Notice that by \eqref{ld.spectral},
\begin{eqnarray*}
I_2&\sim&\int_{\mathcal S}\langle\lambda,s\rangle P\left(\frac H{\|H\|}\in ds,\|H\|>M_n\right)\\
&\sim&P(\|H\|>M_n)\int_{\mathcal S}\langle\lambda,s\rangle\sigma(ds)
\end{eqnarray*}
This shows \eqref{ld.t8.e6} and thus completes the proof.
\end{proof}

Similar calculations as above, for the case $\alpha\ge2$, will show that\\ $S_n/(nM_n^{-1})$ follows LDP with speed $nM_n^{-2}$ and rate function that is the Fenchel-Legendre transform of $\frac12\langle\lambda,D\lambda\rangle$, $D$ being the dispersion matrix of $H$. This is, however, covered in much more generality in Theorem \ref{ld.t9} below, and hence we chose not to include this case in Theorem \ref{ld.t8}.

Cram\'er's Theorem deals with $n^{-1}\sum_{i=1}^nZ_i$ where
$Z_1,Z_2,\ldots$ are i.i.d. random variables. On a finer scale,
$n^{-1/2}\sum_{i=1}^n[Z_i-E(Z_i)]$ possesses a limiting Normal
distribution by the central limit theorem. For $\beta\in(1/2,1)$,
the renormalized quantity $n^{-\beta}\sum_{i=1}^n[Z_i-E(Z_i)]$
satisfies an LDP but always with a quadratic rate function. The
precise statement for this is  known as
moderate deviations; see Theorem 3.7.1 in
\cite{dembo:zeitouni:1998}.
The last result of this section is an analogue of the above result, in the setting of truncated heavy-tailed random variables.

\begin{theorem}[Moderate Deviations]\label{ld.t9}
Suppose that we are in the hard truncation regime, and the sequence $c_n$ satisfies
\begin{equation}\label{hypo1}
n^{1/2}M_nP(\|H\|>M_n)^{1/2}\ll c_n\ll nM_nP(\|H\|>M_n),\mbox{ if }\alpha<2\,,
\end{equation}
\begin{equation}\label{hypo2}
n^{1/2}\ll c_n\ll\frac n{M_n^3P(\|H\|>M_n)},\mbox{ if }2\le\alpha<3\,,
\end{equation}
\begin{equation}\label{hypo3}
n^{1/2}\ll c_n\ll nM_n^{-\delta}\mbox{ for some }\delta>0,\mbox{ if }\alpha=3\,,
\end{equation}
and
\begin{equation}\label{hypo4}
n^{1/2}\ll c_n\ll n,\mbox{ if }\alpha>3\,.
\end{equation}
Then,
$c_n^{-1}(S_n-ES_n)$ follows LDP with speed $\beta_n$ and rate
$\Lambda^*$, the Fenchel-Legendre transform of $\Lambda$, where
$$
\beta_n:=\left\{\begin{array}{ll}\frac{c_n^2}{nM_n^2P(\|H\|>M_n)},&\mbox{if }\alpha<2\,,\\\frac{c_n^2}n,&\mbox{if }\alpha\ge2\,,\end{array}\right.
$$
and
$$
\Lambda(\lambda):=\frac12\langle\lambda,D\lambda\rangle\,.$$ Here,
$D$ is the $d\times d$ matrix with
$$D_{ij}:=\frac2{2-\alpha}\int_{\mathcal S}s_is_j\sigma(ds)$$
if $\alpha<2$ and the dispersion matrix of $H$ if $\alpha\ge2$, which is well defined even when $\alpha=2$ because it has been assumed in that case, that $E\|H\|^2<\infty$.
If, in addition, $D$ is invertible, then $\Lambda^*$ is given by
$$\Lambda^*(x)=\frac12\langle x,D^{-1}x\rangle\,.$$
\end{theorem}

Before proceeding to prove the result, we point out that it is
never vacuous, that is, a sequence $(c_n)$ satisfying the
hypotheses always exists. The existence of a sequence $(c_n)$
satisfying \eqref{hypo1} and \eqref{hypo4} is immediate. Existence
of $(c_n)$ satisfying \eqref{hypo2}  will be clear provided it can
be shown that, if $\alpha\ge2$, then
\begin{equation}\label{ld.t9.new}
n^{1/2}\ll\frac n{M_n^3P(\|H\|>M_n)}\,.
\end{equation}
If $\alpha=2$, then by \eqref{hard}, it follows that
$$
n^{-1/2}M_n^3P(\|H\|>M_n)=o\left(M_n^3P(\|H\|>M_n)^{3/2}\right)=o(1)\,,
$$
the second equality being true because $P(\|H\|>x)=O(x^{-2})$, which is a consequence of the assumption that $E\|H\|^2<\infty$. This shows \eqref{ld.t9.new} when $\alpha=2$. When $\alpha>2$, \eqref{ld.t9.new} will follow because now
\begin{eqnarray*}
M_n^3P(\|H\|>M_n)^{3/2}=o(1)\,.
\end{eqnarray*}
For ensuring the existence of $(c_n)$ satisfying \eqref{hypo3}, observe that for $\delta<\alpha/2$, it holds that
$$
n^{1/2}M_n^{-\delta}\gg n^{1/2}P(\|H\|>M_n)^{1/2}\gg1\,.
$$

\begin{proof}[Proof of Theorem \ref{ld.t9}]

It is easy to see that $\beta_n\longrightarrow\infty$ as $n\longrightarrow\infty$. Thus, in view of the G\"artner-Ellis Theorem, it suffices to show that for all $\lambda\in\bbr^d$,
\begin{equation}\label{ld.t9.eq2}
\lim_{n\to\infty}\beta_n^{-1}\log E\exp\left(\langle\lambda,(M_nb_n)^{-1}(S_n-ES_n)\rangle\right)=\frac12\langle\lambda,D\lambda\rangle\,,
\end{equation}
where
$$b_n:=\left\{\begin{array}{ll}nM_nP(\|H\|>M_n)/c_n,&\alpha<2\\n/(c_nM_n),&\alpha\ge2\,.\end{array}\right.$$
Notice that if $\alpha<3$, then we have that
$$
\frac n{M_n^3P(\|H\|>M_n)}\ll n\,.
$$
By \eqref{hypo2}, \eqref{hypo3} and \eqref{hypo4}, it follows that for all $\alpha\ge2$,
$$
c_n\ll n\,.
$$
Consequently,
\begin{equation}\label{new1}
b_n\gg M_n^{-1}\mbox{ if }\alpha\ge2\,.
\end{equation}
By \eqref{hypo1}, it follows that
\begin{equation}\label{new2}
b_n\gg1\mbox{ if }\alpha<2\,.
\end{equation}
Define
\begin{eqnarray*}
X_n&:=&H\one(\|H\|\le M_n)+\frac
H{\|H\|}(M_n+L)\one(\|H\|>M_n)\,.
\end{eqnarray*}
Let $\xi_n$ be defined by
$$\exp(\langle\lambda,(b_nM_n)^{-1}(X_n-EX_n)\rangle)$$
$$=1+(b_nM_n)^{-1}\langle\lambda,X_n-EX_n\rangle+\frac12(b_nM_n)^{-2}\langle\lambda,(X_n-EX_n)(X_n-EX_n)^T\lambda\rangle+\xi_n\,.$$
Our next claim is that
\begin{eqnarray}
E\exp(\langle\lambda,(b_nM_n)^{-1}(X_n-EX_n)\rangle)
&=&1+\frac12(b_nM_n)^{-2}\langle\lambda,{\mathcal D}(X_n)\lambda\rangle+E\xi_n\nonumber\\
&=&1+\frac12\gamma_n\langle\lambda,D\lambda\rangle(1+o(1))+E\xi_n\,,\label{ld.t9.eq4}
\end{eqnarray}
where
$$\gamma_n:=\left\{\begin{array}{ll}b_n^{-2}P(\|H\|>M_n),&\alpha<2\\b_n^{-2}M_n^{-2},&\alpha\ge2\,.\end{array}\right.$$
Note that \eqref{ld.t9.eq4} follows trivially for the case $\alpha\ge2$. For the case $\alpha<2$, in the proof of Theorem 2.2 of \cite{chakrabarty:samorodnitsky:2009}, it has been shown that as $n\longrightarrow\infty$,
$$\Var(\langle\lambda,X_n\rangle)\sim M_n^2P(\|H\|>M_n)\frac2{2-\alpha}\int_{\mathcal S}\langle\lambda,s\rangle^2\sigma(ds)\,,$$
which essentially means \eqref{ld.t9.eq4}.

Clearly, $n\gamma_n=\beta_n$, and by \eqref{new1} and \eqref{new2}, it follows that
$$
\lim_{n\to\infty}\gamma_n=0\,.
$$
Hence all that needs to be shown for \eqref{ld.t9.eq2} is $E\xi_n=o(\gamma_n)$ as $n\longrightarrow\infty$. By Taylor's Theorem, there exists $C<\infty$ so that
\begin{eqnarray*}
|\xi_n|&\le&C(b_nM_n)^{-3}\|X_n-EX_n\|^3\exp\left\{C(b_nM_n)^{-1}\|X_n-EX_n\|\right\}\\
&\le&C(b_nM_n)^{-3}\|X_n-EX_n\|^3\exp\left\{Cb_n^{-1}\left(4+\frac{L+E(L)}{M_n}\right)\right\}\\
&\le&8C(b_nM_n)^{-3}\left(\|X_n\|^3+\|EX_n\|^3\right)\exp\left\{Cb_n^{-1}\left(4+\frac{L+E(L)}{M_n}\right)\right\}\,.
\end{eqnarray*}
Thus,
$$E|\xi_n|=O\left((b_nM_n)^{-3}E\left[\left(\|X_n\|^3+\|EX_n\|^3\right)\exp(CL/b_nM_n)\right]\right)\,.$$
Note that
\begin{eqnarray*}
&&E\left[\|X_n\|^3\exp(CL/b_nM_n)\right]\\
&=&E\left[\|H\|^3\one(\|H\|\le M_n)\right]E\left[\exp(CL/b_nM_n)\right]\\
&&+P(\|H\|>M_n)E\left[(M_n+L)^3\exp(CL/b_nM_n)\right]\\
&=&O(1)E\left[\|H\|^3\one(\|H\|\le M_n)\right]+O\left(M_n^3P(\|H\|>M_n)\right)\,.
\end{eqnarray*}
Also,
\begin{eqnarray*}
&&\|EX_n\|^3E\left[\exp(CL/b_nM_n)\right]\\
&=&O(E(\|X_n\|^3))\\
&=&O\left(E\left[\|H\|^3\one(\|H\|\le M_n)\right]+M_n^3P(\|H\|>M_n)\right)\,,
\end{eqnarray*}
the last step following by similar calculations as above. Thus,
$$E\xi_n=$$
\begin{equation}\label{ld.t9.eq3}
O\left\{(b_nM_n)^{-3}\left(E\left[\|H\|^3\one(\|H\|\le M_n)\right]+M_n^3P(\|H\|>M_n)\right)\right\}\,.
\end{equation}
We claim that for all $\alpha$,
\begin{equation}\label{ld.t9.eq1}
P(\|H\|>M_n)=o(b_n^3\gamma_n)\,.
\end{equation}
This is immediate by \eqref{new2} if $\alpha<2$, and by \eqref{hypo2} if $2\le\alpha<3$. When $\alpha=3$,
\begin{eqnarray*}
P(\|H\|>M_n)\ll M_n^{-3+\delta}
\ll \frac n{c_n}M_n^{-3}
=b_n^3\gamma_n\,,
\end{eqnarray*}
the second inequality following from \eqref{hypo3}. Thus, \eqref{ld.t9.eq1} holds when $\alpha=3$.
For the case $\alpha>3$, \eqref{hypo4} implies \eqref{ld.t9.eq1}.

If $\alpha<3$, then by Karamata's Theorem,
$$E\left[\|H\|^3\one(\|H\|\le M_n)\right]=O(M_n^3P(\|H\|>M_n))\,.$$
Hence by \eqref{ld.t9.eq3} and \eqref{ld.t9.eq1}, it follows that
 $E\xi_n=o(\gamma_n)$ for the case $\alpha<3$. If $\alpha=3$, then
$$E\left[\|H\|^3\one(\|H\|\le M_n)\right]=o(M_n^\delta)=o(b_n^3M_n^3\gamma_n)\,.$$
Using \eqref{ld.t9.eq3} and \eqref{ld.t9.eq1}, this shows that $E\xi_n=o(\gamma_n)$ for the case $\alpha=3$. When $\alpha>3$,
$$E\left[\|H\|^3\one(\|H\|\le M_n)\right]=O(1)=o(b_n^3M_n^3\gamma_n)\,,$$ and this completes the proof.
\end{proof}

\section{Conclusions}\label{sec:con} The proofs of the results in Section \ref{sec:ld.st} make it clear that in the soft truncation regime, the idea leading to the investigation of the large deviation behavior is similar to that in the case of untruncated heavy-tailed distributions, as studied in \cite{hult:lindskog:mikosch:samorodnitsky:2005}, for example. The argument in the untruncated case is based on showing that the partial sum is large ``if and only if'' exactly one of the summands is large, while in the softly truncated case, it was showed that the partial sum is large ``if and only if'' the sum of a fixed number of them is large. The similarity between the two situations is clear. The results of Section \ref{sec:ld.ht} show that the large deviation analysis in the case where the tails are truncated hard follow the same route as that for i.i.d. random variables with exponentially light tails, namely the G\"artner-Ellis Theorem. Thus, the analysis carried out in this paper provides the following answer to the question posed in Section \ref{sec:intro}: when the growth rate of the truncating threshold is fast enough so that the model is in the soft truncation regime, the effect of truncating by that is negligible, whereas when the same is slow enough so that the model is in the hard truncation regime, the effect is significant to the point that the model then behaves like a light-tailed one.

\section{Acknowledgements} The author is immensely grateful to his adviser Gennady Samorodnitsky for some helpful discussions. He also acknowledges the comments and suggestions of two anonymous referees and an Associate Editor, which helped improve the presentation significantly.


\end{document}